\documentclass[a4paper,10pt,reqno]{amsart}
\usepackage[UKenglish]{babel}
\usepackage{amsmath}
\usepackage{amssymb}
\usepackage{verbatim}
\usepackage{stackrel}
\usepackage[arrow, matrix, curve]{xy} %for diagrams of implications

\usepackage[shortlabels]{enumitem}
\usepackage{caption} 
\usepackage{tabularray}

\usepackage{comment,cite}	
\usepackage{hyperref}
\hypersetup{
    colorlinks=true,
    linkcolor=blue,
    citecolor=red,
    filecolor=magenta,      
    urlcolor=cyan,
    pdftitle={admissibility},
    linktocpage=true,
}

\usepackage[utf8]{inputenc}

\usepackage{amssymb}
\usepackage{amsmath}
\usepackage{bbm}
\usepackage{verbatim, stackrel}

\usepackage[shortlabels]{enumitem}
\usepackage{marginnote}
\usepackage{color}

\usepackage{stmaryrd}
\usepackage{tikz}
\usepackage{tikz-cd}

\newcommand{\bbC}{\mathbb{C}}

\newcommand{\bbN}{\mathbb{N}}

\newcommand{\bbR}{\mathbb{R}}

\newcommand{\calD}{\mathcal{D}}

\newcommand{\calF}{\mathcal{F}}

\newcommand{\calL}{\mathcal{L}}
\newcommand{\calM}{\mathcal{M}}

\DeclareMathOperator{\one}{{\mathbbm{1}}} % constant function with value one
\DeclareMathOperator{\re}{Re} % real part
\newcommand{\argument}{\mathord{\,\cdot\,}} % argument dot for functions (with correct spacing)
\newcommand{\dx}{\;\mathrm{d}} % differential (for use at the end of integrals)
\newcommand{\dxMedium}{\nobreak\hspace{.1em plus .08333em}\mathrm{d}} % differential smaller space
\newcommand{\dxShort}{\mathrm{d}} % differential without extra space (for use in differential quotients)
\DeclareMathOperator{\linSpan}{span} % linear span
\newcommand{\norm}[1]{\left\lVert #1 \right\rVert} % norm
\newcommand{\modulus}[1]{\left\lvert #1 \right\rvert} % modulus
\newcommand{\duality}[2]{\left\langle#1\, ,\, #2\right\rangle} % duality / scalar product
\DeclareMathOperator{\dom}{dom} % domain (of an unbounded operator)
\DeclareMathOperator{\Ima}{Rg} % image / range of a mapping
\newcommand\restrict[1]{\raisebox{-.5ex}{$|$}_{#1}} %restriction of a map

\newcommand{\Reg}{\operatorname{Reg}} % regulated
\renewcommand{\div}{\operatorname{div}} % divergence

\newcommand{\spec}{\sigma} % spectrum
\newcommand{\resSet}{\rho}
\newcommand{\Res}{R} % resolvent
\newcommand{\spb}{s} % spectral bound
\newcommand{\gbd}{\omega_0} % growth bound of a C_0-semigroup

\theoremstyle{plain}

\theoremstyle{definition}
\newtheorem{definition}{Definition}[section]
\newtheorem{remark}[definition]{Remark}

\newtheorem*{remark*}{Remark}
\newtheorem*{remarks*}{Remarks}
\newtheorem{example}[definition]{Example}
\newtheorem{examples}[definition]{Examples}
\newtheorem{proposition}[definition]{Proposition}
\newtheorem{lemma}[definition]{Lemma}
\newtheorem{theorem}[definition]{Theorem}
\newtheorem{corollary}[definition]{Corollary}
\newtheorem{assumption}[definition]{Assumption}

\numberwithin{equation}{section} % enumerate formulas within sections

\title{Limit-case admissibility for positive infinite-dimensional systems}

\author{Sahiba Arora}
\address{Sahiba Arora, Department of Applied Mathematics, University of Twente, 217, 7500 AE, Enschede, the Netherlands}
\email{sahiba.arora@math.uni-hannover.de}

\author{Jochen Glück}
\address{Jochen Glück, University of Wuppertal, School of Mathematics and Natural Sciences, Gaußstr.\ 20, 42119 Wuppertal, Germany}
\email{glueck@uni-wuppertal.de}

\author{Lassi Paunonen}
\address{Lassi Paunonen, Mathematics Research Centre, Tampere University, Tampere, Finland}
\email{lassi.paunonen@tuni.fi}

\author{Felix L. Schwenninger}
\address{Felix L. Schwenninger, Department of Applied Mathematics, University of Twente, 217, 7500 AE, Enschede, the Netherlands}
\email{f.l.schwenninger@utwente.nl}

\date{\today}

\begin{document}

\begin{abstract}
    In the context of positive infinite-dimensional linear systems, we systematically study $L^p$-admissible control and observation operators with respect to the limit-cases $p=\infty$ and $p=1$, respectively. This requires an in-depth understanding of the order structure on the extrapolation space $X_{-1}$, which we provide. These properties of $X_{-1}$ also enable us to discuss when zero-class admissibility is automatic.  While those limit-cases are the weakest form of admissibility on the $L^p$-scale, it is remarkable that they sometimes directly follow from order theoretic and geometric assumptions. Our assumptions on the geometries of the involved spaces are minimal.
\end{abstract}

\maketitle

\section{Introduction}
In this paper, we study the boundedness of linear operators 
\begin{equation*}
        \Phi_{\tau}:L^{p}([0,\tau];U)\to X, \qquad  
        u\mapsto x(\tau), \quad \tau>0
\end{equation*}
for some $p\in[1,\infty]$, arising in boundary control systems \cite{Greiner1987,Salamon1987,TucsnakWeiss2009} of the form 
\begin{align*}
    \dot{x}(t)={}&\mathfrak{A}x(t),\quad t>0,\quad x(0)=0,\\
    \mathfrak{B}x(t)={}&u(t).
\end{align*}
Here $\mathfrak{A}:\dom \mathfrak{A}\subset X\to X$ and $\mathfrak{B}:\dom \mathfrak{A}\to U$ are linear operators acting on Banach spaces $X$ and $U$. Under the assumptions that the restriction $A=\mathfrak{A}|_{\ker\mathfrak{B}}$ generates a $C_{0}$-semigroup and that $\mathfrak{B}$ has a bounded right-inverse, the above-mentioned boundedness can be rephrased in terms of \emph{admissible operators}, going back to Weiss \cite{Weiss1989a,Weiss1989b}. We also refer to \cite{EngelKramar17}, where the equivalent viewpoint used in this introduction was first taken. Indeed, the boundedness of $\Phi_{\tau}$ is equivalent to the property that for some $\lambda$ in the resolvent set $\rho(A)$ of $A$, the operator
\begin{equation*}
    \widetilde{\Phi}_{\lambda,\tau}:L^{p}([0,\tau];U)\to X,\qquad u\mapsto \int_{0}^{\tau}T(\tau-s)B_{\lambda}u(s)\mathrm{ds}
\end{equation*}
has range in $\dom A$ for some (hence all) $\tau>0$, where $B_{\lambda}=\left(\mathfrak{B}|_{\ker( \mathfrak{A}-\lambda)}\right)^{-1}$ is a well-defined bounded operator from $U$ to $X$,  see e.g.~\cite{Greiner1987} or \cite[Remark~2.7]{EngelKramar17}.  In that case, 
$\Phi_{\tau}u=x(\tau)=(A-\lambda)\widetilde{\Phi}_{\lambda,\tau}u$ for any $\lambda\in\rho(A)$.

Admissible operators are an indispensable tool in the study of infinite-dimensional systems, particularly in the context of well-posed systems \cite{JacobPartington2004,TucsnakWeiss2014,Staffans2005}. 
The definition of admissible (control) operators is usually given for systems in state-space form, 
\begin{equation}\label{eq:SSsys}
    \dot{x}(t)=Ax(t)+Bu(t),\quad x(0)=0;
\end{equation}
This class includes -- via the choice $B=(\lambda-A_{-1})B_{\lambda}$ with $A_{-1}$ being the extension of $A$ to the extrapolation space $X_{-1}$ corresponding to the semigroup (generated by $A$) -- boundary control systems if $B$ is only required to map $U$ to  $X_{-1}$, see \cite[Section~10.1]{TucsnakWeiss2009} or \cite[Section~2]{Schwenninger2020}. Admissibility for $p=2$ confines a rich theory in the Hilbert space setting, see, for instance, \cite{JacobPartington2001,Staffans2005,TucsnakWeiss2009}. While several results for $p\in(1,\infty)\setminus\{2\}$ exist  \cite{Haak2004,HaakLeMerdy2005,JacobPartingtonPott2014,Weiss1989b}, the case $p=\infty$ has been studied systematically only recently  \cite{JacobNabiullinPartingtonSchwenninger2018,JacobSchwenningerZwart2019, Wintermayr2019}. Note that the methodology for studying admissibility heavily relies on the specific context, such as whether $X$ is a Hilbert space or if the semigroup is analytic or extendable to a group.

In this note, we focus on the case $p=\infty$ and add another additional structure to the setting: positivity of the semigroup and the operator $B$ (or $B_{\lambda}$) -- in the sense of ordered Banach spaces. In fact, the positivity of $B$ can be characterised by the positivity of the operators $B_\lambda$ (Section~\ref{subsec:positive-B-via-boundary-operator}). 
While positivity is a well-studied concept for operator semigroups, its relation to admissibility, or general infinite-dimensional systems theory, is less understood. 
In \cite[Proposition~4.3]{EngelKramar17}, it is shown that positivity of $\Phi_{\tau}$ and $B_{\lambda}$ for sufficiently large $\lambda$  are equivalent provided that admissibility is assumed. 
Controllability of positive systems is studied in \cite{EngelKramar17, Gantouh2023}. 
Admissibility criteria for positive $B$ are given in \cite[Chapter~4]{Wintermayr2019} and \cite[Theorem~2.1]{Gantouh2022a}.
The first reference focuses on $L^\infty$- and $C$-admissibility by imposing assumptions mainly on the space $U$. 
In the second reference, a strong assumption on the semigroup is made which implies $L^{1}$-admissibility for all positive $B$'s. 
In Appendix~\ref{appendix:al-characterisation}, we prove that under a reasonable compactness condition, this assumption can only be satisfied if $X$ is an $L^{1}$-space. 
Our aim -- in contrast -- is to list assumptions on the state space such that admissibility follows from the assumed positivity in a rather automatic fashion. First results in this direction were derived by Wintermayr \cite{Wintermayr2019}. It should be noted that even the notion of positivity for control operators $B$ stemming from~\eqref{eq:SSsys} is nontrivial as positivity in the space $X_{-1}$ needs to be defined suitably \cite{BatkaiJacobVoigtWintermayr2018}. Motivated by this, we devote Section~\ref{sec:extrapolation-space} to a refined study of the order structure of $X_{-1}$, which is of interest in its own right. Under mild assumptions, the positive cone $X_+$ turns out to be a face in the cone $X_{-1,+}$. 
Moreover, we show that if the semigroup satisfies a suitable ultracontractivity assumption, then $X_{-1,+}$ is contained in an interpolation space between $X$ and $X_{-1}$ -- which immediately yields admissibility (Section~\ref{sec:L^r-admissibility}).

In addition, we also study  the formally dual notion of \emph{$L^1$-admissible observation operators} $C:\dom A\to Y$, for a Banach space $Y$, meaning that the mapping
\[
    \Psi_{\tau}:\dom A\to L^{1}([0,\tau];Y), x_{0}\mapsto CT(\cdot)x_{0}
\]
extends to a bounded linear operator on $X$. 
We present two types of results:~first, we assume conditions on the order structure of the spaces $X$ and $Y$, as well as positivity of the semigroup and observation operator. Another result, which is independent of any positivity assumptions and thus of interest in its own right, is Theorem~\ref{thm:zero-class-observation-sufficient}, stating that zero-class $L^{1}$-admissibility is automatic from $L^1$-admissibility if $X$ is reflexive and $Y$ an $\mathrm{AL}$-space.
The case of $L^{\infty}$-admissible control operators is treated in Section~\ref{sec:admissibility-control}, which also generalises several results by Wintermayr  \cite{Wintermayr2019}. The derived results relate to input-to-state stability of infinite-dimensional systems, a notion that has seen tremendous interest within the past decade; see \cite{Mironchenko2023} for an overview. 

We point out that in the limiting cases of the H\"older conjugates $p=1$ and $p=\infty$ the duality between admissible control and observation operators is subtle and does not allow for a one-to-one translation between them. 

In the last section, Section~\ref{sec:perturbations}, we apply our findings to well-known perturbation results. Certain assumptions of Sections~\ref{sec:admissibility-observation} and~\ref{sec:admissibility-control} are elaborated on in Appendix~\ref{sec:factorisation}. In the rest of the introduction, we recall basic concepts and fix our notations.

\subsection*{Banach spaces and operators}

The closed unit ball of a Banach space $X$ is denoted by $B_X$.
The space of bounded linear operators between two Banach spaces $X$ and $Y$ will be denoted by $\calL(X,Y)$. For a closed linear operator $A$, we denote its domain, kernel, and range by $\dom A$, $\ker A$, and $\Ima A$, and we write $A'$ for the dual operator acting on the dual space $X'$. The resolvent set and spectrum of $A$ are as usual denoted by $\resSet(A)$ and $\spec(A)$ and $\Res(\lambda,A) := (\lambda-A)^{-1}$ is the resolvent of $A$ at a point $\lambda \in \rho(A)$. 
The spectral bound of $A$ is given by $\spb(A) := \sup \{\re \lambda: \lambda \in \spec(A)\}$.

\subsection*{Operator semigroups}

We assume the reader is familiar with the theory of $C_0$-semigroups, for which we refer to the monograph \cite{EngelNagel2000}. The interpolation and extrapolation spaces associated with a $C_0$-semigroup play an important role in the context of admissibility. Let $(T(t))_{t\ge 0}$ be a $C_0$-semigroup on a Banach space $X$ with generator $A$ and fix $\lambda\in \resSet(A)$. The corresponding \emph{interpolation space}  
$X_1:= (\dom A, \norm{\argument}_1)$, 
where $\norm{\argument}_1$ is the graph norm,
and \emph{extrapolation space} -- defined as the completion
$X_{-1}:= (X, \norm{\argument}_{-1})^{\sim}$, where $\norm{\argument}_{-1}:=\norm{\Res(\lambda,A)\argument}$,
are both Banach spaces. 
For different choices of $\lambda$, the norms on $X_{-1}$ are equivalent. Further, $(T(t))_{t\ge 0}$ extends uniquely to a $C_0$-semigroup on $X_{-1}$, denoted by $(T_{-1}(t))_{t\ge 0}$. The generator $A_{-1}$ of $(T_{-1}(t))_{t\ge 0}$ has domain $\dom A_{-1}=X$ and is the unique extension of $A$ to a bounded operator from $X$ to $X_{-1}$. Important properties that we use tacitly throughout are the isomorphisms $(X_1)'=(X')_{-1}$ and $(X_{-1})'=(X')_1$, where $(X')_{-1}$ and $(X')_1$ refer to the dual semigroup on $X'$; see \cite[Corollary~3.1.17]{vanNeerven1992}.

\subsection*{Ordered Banach spaces and Banach lattices}

In concrete examples, the state space of a system is usually a function space such as $L^p$ and hence a Banach lattice. 
On the other hand, natural candidates for input or output spaces are often spaces of differentiable functions -- for instance, Sobolev spaces on the boundary -- that are ordered Banach spaces but not Banach lattices. 
Moreover, even if one starts with a Banach lattice $X$, the canonical order on the extrapolation space $X_{-1}$ may not render $X_{-1}$ a Banach lattice \cite[Example~5.1]{BatkaiJacobVoigtWintermayr2018}. 
For these reasons, we find it natural to set up the entire theory within the general framework of ordered Banach spaces.

A \emph{wedge} in a Banach space $X$ is a non-empty set $X_+ \subseteq X$ such that $\alpha X_+ + \beta X_+ \subseteq X_+$ for all scalars $\alpha, \beta \ge 0$.
We use the short-hand $-X_+:= \{-x : \, x \in X_+\}$.
A Banach space $X$ together with a closed wedge $X_+ \subseteq X$ is called a \emph{pre-ordered Banach space}.
The set $X_+$ is called the \emph{positive wedge} of $X$ and it induces a natural pre-order (i.e., a reflexive and transitive relation) on $X$:~$x\le y$ if and only $y-x \in X_+$.
The wedge $X_+$ is called a \emph{cone} if $X_+ \cap -X_+ = \{0\}$, which is equivalent to the pre-order $\le$ being anti-symmetric and thus a partial order. 
If $X_+$ is a cone, then we call $X$ an \emph{ordered Banach space} and $X_+$ the \emph{positive cone} of $X$.

Let $X$ be a pre-ordered Banach space. 
The set $X_+ - X_+ = \{x-y : \, x,y \in X_+ \}$ is a vector subspace of $X$ and coincides with the linear span of $X_+$.
The wedge $X_+$ is called \emph{generating} if $X=X_+ - X_+$ and \emph{normal} if there are $M\ge 1$ such that for each $x,y \in X_+$ the inequality $x\le y$ implies $\norm{x}\le M\norm{y}$. 
Every normal wedge is automatically a cone.
If $X$ is a pre-ordered Banach space we endow $\linSpan X_+ = X_+ - X_+$ with the norm
\begin{equation}
    \label{eq:norm-on-span-of-cone}
    \norm{x}_{X_+-X_+} 
    := 
    \inf\{\norm{y} + \norm{z} : \, y,z \in X_+ \text{ and } x = y-z\}.
\end{equation}
It is a complete norm on $X_+ - X_+$ stronger than the norm induced by $X$ \cite[Lemma~2.2]{ArendtNittka2009} and both norms coincide on $X_+$ (to avoid potential confusion, we note that a pre-ordered Banach space is called an ordered Banach space in \cite{ArendtNittka2009}). 
Thus, $\norm{\argument}_{X_+-X_+}$ turns $X_+-X_+$ into a pre-ordered Banach space and if the wedge $X_+$ is normal with respect to the norm on $X$, then it is also normal with respect to $\norm{\argument}_{X_+-X_+}$.

A subset $S$ of a pre-ordered Banach space $X$ is said to be \emph{order-bounded} if there exist $x,z\in X$ such that $S$ is contained in the so-called \emph{order interval} $[x,z] := \{y \in X : \, x \le y \le z\}$. 
A non-empty set $C\subseteq X_+$ is called a \emph{face} of $X_+$ if $[0,x]\subseteq C$ for all $x\in C$ and $C$ is also a wedge.
A subspace $V$ of $X$ is said to be \emph{majorizing} in $X$ if for each $x \in X$, there exists $v \in V$ such that $x \le v$.

If $X$ is a pre-ordered Banach space, then $X'_+ := \{x' \in X' : \, \langle x', x \rangle \ge 0 \text{ for all } x \in X_+\}$ is called the \emph{dual wedge} of $X_+$; 
it turns the dual space $X'$ into a pre-ordered Banach space.
It follows from the Hahn-Banach separation theorem that $\linSpan X_+$ is dense in $X$ if and only if $X'_+$ is a cone. 
Conversely, $\linSpan X'_+$ is weak${}^*$-dense in $X'$ if and only if $X_+$ is a cone. 
One can prove that $X_+$ is generating if and only if $X'_+$ is normal and that $X_+$ is normal if and only if $X'_+$ is generating \cite[Theorems~4.5 and 4.6]{KrasnoselskiiLifshitsSobolev1989}; 
the reference states the results for ordered Banach spaces there, but they remain true for pre-ordered Banach spaces.
Our main interest throughout the manuscript is in ordered (rather than pre-ordered) Banach spaces, since those spaces typically occur in applications. 
But since the dual and the bidual of an ordered Banach space need only be pre-ordered Banach spaces in general and since we make extensive use of duality theory in Section~\ref{sec:order-properties}, it makes the theory easier and clearer if one has the terminology of pre-ordered Banach spaces available.

If $e$ is a positive element of an ordered Banach space $X$, then the \emph{principal ideal generated by $e$} is defined as
\begin{equation}
    \label{eq:principal-ideal}
    X_e:= \bigcup_{\lambda>0} [-\lambda e,\lambda e].
\end{equation}
If the cone of $X$ is normal, then by \cite[Theorem~2.60]{AliprantisTourky2007}, $X_e$ becomes an ordered Banach space when equipped with the \emph{gauge norm}
$
    \norm{x}_e := \inf \{\lambda>0: x \in [-\lambda e,\lambda e] \}.
$
An element $e \in X_+$ is called a \emph{unit} of $X$ if $X_e = X$. In particular, $e$ is always a unit of the ordered Banach space $X_e$ itself. One can show that $e$ is a unit if and only if it is an interior point of $X_+$, even if $X_+$ is not normal, see e.g.\ \cite[Proposition~2.11]{GlueckWeber2020} for details. 
Whenever $X$ has a unit $e$, say with norm $1$, we shall endow $X$ with the equivalent norm $x \mapsto \max\{\norm{x}, \norm{x}_e\}$; in this case $e$ is the largest element of the closed unit ball and hence, the closed unit ball is \emph{upwards directed}, i.e., for all $x_1,x_2 \in B_X$ there exists $x\in B_X$ such that $x_1,x_2\le x$. 
As a consequence, the open unit ball of $X$ is also upwards directed, a property that will occur in several of our results.
\begin{center}
    \begin{table}
        \centering\small
        \begin{tblr}{%
            vlines, hlines,
             colspec = {ccc},
            }
            \SetCell{m, 2.5cm} \textbf{Properties} &
            \SetCell{m, 2cm} \textbf{Banach lattice terminology} & 
            \SetCell[c=2]{m, 5.5cm} \textbf{Typical example} 
            \\
            \SetCell{m, 2.5cm} {\textbf{Norm-bounded increasing nets are norm-convergent}} & 
            \SetCell{m, 2cm} KB-space &
            \SetCell[c=2]{m, 5.5cm}   $L^p$ for $p\in [1,\infty)$ 
            \\ 
            \SetCell{m, 3cm} \textbf{Cone is a face of bidual wedge} & 
            \SetCell{m, 2cm} Order continuous norm &
            \SetCell[c=2]{m, 5.5cm} $L^p$ for $p\in [1,\infty)$ and $c_0$ 
            \\
            \SetCell{m, 3cm} \textbf{Norm is additive on the cone} & 
            \SetCell{m, 2cm} AL-space &
            \SetCell[c=2]{m, 5.5cm} $L^1$ and $\calM(\Omega)$ for measurable $\Omega$ &
            \\
            \SetCell[r=2]{m, 3cm} \textbf{Open unit ball is upwards directed} & 
            \SetCell[r=2]{m, 2cm} AM-space &
            \SetCell{m, 1cm} with unit &
            \SetCell{c,m,4.5cm} {$\mathrm C(K)$ for compact $K$ and $L^\infty$}\\
            & 
            &
            \SetCell{m, 1cm} general &
            \SetCell{c,m, 4.5cm} $\mathrm C_0(L)$ for locally compact $L$\\
        \end{tblr} 
        \caption{Important properties of ordered Banach spaces}
        \label{table:properties}
    \end{table}
\end{center}
Banach lattices form a special class of ordered Banach spaces.
An element $z\in X$ is said to be the \emph{supremum} or the least upper bound of $x,y\in X$-- which we denote by $z=\sup\{x,y\}$ when it exists -- if $z \ge x,y$ and if $u \ge x,y$ for $u\in X$ implies $u \ge z$.
Analogously, one can define the infimum$\inf\{x,y\}$ of $x$ and $y$ as their greatest lower bound whenever it exists.
An ordered Banach space $X$ is called a \emph{Banach lattice} if any two elements have a supremum (equivalently:~infimum) and 
$
    \modulus{x}\le \modulus{y} \Rightarrow \norm{x} \le \norm{y}
$,
where $\modulus{x}=\sup\{x,-x\}$. The cone of a Banach lattice is always generating and normal. 
Common examples of Banach lattices  are the $L^p$ spaces for $p\in [1,\infty]$ and the spaces of continuous functions with the supremum norm. The theory of ordered Banach spaces and Banach lattices is classical, see \cite{Schaefer1974, Meyer-Nieberg1991, AliprantisBurkinshaw2006, AbramovichAliprantis2001}. 
Many ordered Banach spaces $X$ have additional geometric properties and if $X$ is even a Banach lattice, those properties often have special names. 
As those properties are useful in our results, we summarise them -- along with their name in the case that $X$ is a Banach lattice and  common examples --  in Table~\ref{table:properties}.

Let $T:X\to Y$ be a linear map between pre-ordered Banach spaces $X$ and $Y$. 
We say that $T$ is \emph{positive} and write $T\ge 0$ if $TX_+\subseteq Y_+$. 
By $\calL(X,Y)_+$ we denote the positive, bounded linear operators in $\calL(X,Y)$. If $X_+$ is generating, then positivity of $T$ implies boundedness \cite[Theorem~2.8]{ArendtNittka2009}. 
An operator $T \in \calL(X,Y)$ is positive if and only if its dual operator $T' \in \calL(Y',X')$ is positive.
A functional $\varphi \in X'$ is called \emph{strictly positive} if its kernel contains no positive non-zero element. A $C_0$-semigroup $(T(t))_{t\ge 0}$ on $X$ is called \emph{positive} if each operator $T(t)$ is positive. For the theory of positive $C_0$-semigroups, we refer the reader to  \cite{Nagel1986, BatkaiKramarRhandi2017, BattyRobinson1984}. 
A linear operator $J:X\to Y$ between pre-ordered Banach spaces is called \emph{bipositive} if for each $x \in X$ we have $Jx \ge 0$ if and only if $x\ge 0$. 
Intuitively, if there exists a bipositive map $J \in \calL(X,Y)$ and $J$ is injective (see Proposition~\ref{prop:span-determined}(a)), we can consider $X$ as a subspace of $Y$, endowed with the pre-order inherited from $Y$.
Thus we shall sometimes say that $X_+$ is a face of $Y_+$ to mean that $J(X_+)$ is a face of $Y_+$.

\subsection*{Complexifications}

Ordered Banach spaces and Banach lattices are theories over the real field. 
However, when ones uses spectral theory or analytic semigroups it is natural to work with complex scalars. 
To this end, one can use \emph{complexifications} of real Banach spaces as they are, for instance, described in \cite{MunozSarantopoulosTonge1999}, \cite[Appendix~C]{Glueck2016b}, and -- specifically for Banach lattices -- \cite[Section~II.11]{Schaefer1974}. 
The details do not cause any problems in our setting, so whenever we discuss spectral properties of a linear operator $A$ between two real Banach spaces $X$ and $Y$, we shall tacitly mean the property of the extension of $A$ to complexifications of $X$ and $Y$.

\section{Order properties of the extrapolation space $X_{-1}$}
    \label{sec:extrapolation-space}

To utilise any order assumptions on the  control operator $B\in \calL(U, X_{-1})$ that occurs in the system~\eqref{eq:SSsys}, it is crucial to understand the order structure of the space $X_{-1}$. In Section~\ref{sec:order-properties}, we first show in an abstract setting how order properties can be transferred between non-isomorphic ordered Banach spaces. Subsequently, in Section~\ref{sec:order-properties-corollaries}, we describe an order on $X_{-1}$ and show how various order properties of $X$ are carried over to $X_{-1}$.

\subsection{Transferring properties between (pre-)ordered Banach spaces}
    \label{sec:order-properties}

The setting of this subsection is two (pre-)ordered Banach spaces $Z$ and $X$ that are related via certain positive operators. 
We start with some properties of bipositive operators.

\begin{proposition}
    \label{prop:span-determined}
    Let $J:X\to Z$ be a bounded operator between pre-ordered Banach spaces $X$ and $Z$. 
    \begin{enumerate}[\upshape (a)]
        \item If $X$ is an ordered Banach space and $J$ is bipositive, then  $J$ is injective. 

        \item 
        The map $J'$ is bipositive if and only if $J(X_+)$ is a dense subset of $Z_+$.

        \item 
        If $Z_+$ is generating in $Z$ and $J$ is bipositive, injective, and has majorizing range,
        then $J'(Z'_+) = X'_+$
    \end{enumerate}
\end{proposition}

\begin{proof}
    (a)
    Let $x \in X$ such that $Jx = 0$. 
    Then $Jx \ge 0$ and $J(-x) \ge 0$, so the bipositivity of $J$ implies that $x \ge 0$ and $-x \ge 0$. 
    Hence, $x = 0$ since $X_+ \cap -X_+ = \{0\}$.
    
    (b) 
    Let $J'$ be bipositive. 
    As $J'$ is positive, so is $J$ and hence, $J(X_+) \subseteq Z_+$.
    Assume that the inclusion is not dense. 
    By the Hahn-Banach separation theorem, there exists $z \in Z_+$ and $z' \in Z'$ such that $\duality{z'}{z} < \duality{z'}{Jx} = \duality{J'z'}{x}$ for all $x \in X_+$. 
    Taking $x = 0$ we obtain $\duality{z'}{z}< 0$, so $z' \not\in Z'_+$. 
    Replacing $x \in X_+$ with $nx$ for large $n$ we get that $0 \le \langle J'z',x\rangle$, so $J'z'\ge 0$, contradicting the bipositivity of $J'$.
    
    Conversely, let $J(X_+)$ be  dense in $Z_+$.  In particular $J$ and hence, $J'$ is positive.
    Now let $z' \in Z'$ be such that $J'z' \ge 0$. 
    Therefore, $\langle z', Jx \rangle \ge 0$ for all $x \in X_+$. 
    By density of $J(X_+)$ in $Z_+$, we conclude that $\langle z', z \rangle \ge 0$ for all $z \in Z_+$, so $z' \ge 0$.

    (c)
    Note that $J'(Z'_+) \subseteq X'_+$, so we only need to show the reverse inclusion. For this, let $x' \in X'_+$. 
    The linear mapping $\varphi: J(X) \to \bbR$, $v \mapsto \langle x', J^{-1}v \rangle$ is well-defined and positive by the injectivity and bipositivity of $J$.
    Since $J(X)$ is by assumption majorizing in $Z$, the Kantorovich extension theorem \cite[Theorem~1.36]{AliprantisTourky2007} (which is formulated there for ordered vector spaces, but can be checked to also hold in the pre-ordered case) implies that $\varphi$ extends to a positive linear functional $z': Z \to \bbR$. 
    As $Z_+$ is generating,  $z'$ is automatically continuous \cite[Theorem~2.8]{ArendtNittka2009}, i.e., $z' \in Z'_+$. 
    Moreover, the equality $\langle z', Jx \rangle = \varphi(Jx) = \langle x', x \rangle$ for each $x \in X$ implies that $x'=J'z'\in J'(Z'_+)$.
\end{proof}

\begin{example}[Dual cone of $\mathrm C^k(\overline{\Omega})$]
    For $k\in \bbN$ and a bounded domain $\Omega \subseteq \bbR^d$, let $\mathrm C^k(\overline{\Omega})$ denote the space of all $k$-times continuously differentiable functions on $\Omega$ whose derivatives up to order $k$ extend continuously to $\overline{\Omega}$. 
    This is an ordered Banach space with the usual pointwise order and norm. We show that its dual cone is (more precisely, can be identified with) the set of all positive finite Borel measures on $\overline{\Omega}$. 

    Indeed, the canonical embedding $J: \mathrm C^k(\overline{\Omega}) \to \mathrm C(\overline{\Omega})$ is bipositive
    and has majorizing range because it contains the constant function $\one$. Moreover, by employing the density of $\mathrm C^k(\overline{\Omega})$ in $\mathrm C(\overline{\Omega})$ (which is true by the Stone-Weierstraß approximation theorem) and by using shifts by small multiples of the constant function $\one$, we even have that the cone $\mathrm C^k(\overline{\Omega})_+$ is dense in the cone $\mathrm C(\overline{\Omega})_+$.    
    So $J'$ is bipositive and injective according to Proposition~\ref{prop:span-determined}. 
    Hence, $J'$ is a bipositive embedding of $\mathrm C(\overline{\Omega})'$ (which can be identified with the space of finite Borel measures on $\overline{\Omega}$) into $\big(\mathrm C^k(\overline{\Omega})\big)'$. 
    Hence, Proposition~\ref{prop:span-determined}(c) shows that the dual cone of $\mathrm C^k(\overline{\Omega})$ consists of (to be precise, can be identified via $J'$ with) the set of all positive finite Borel measures on $\overline{\Omega}$. 
\end{example}

In our next result, we show that the first two properties in Table~\ref{table:properties} are preserved by the inverse of a positive bounded bijection (which might not be bipositive).

\begin{theorem}
    \label{thm:inherited-properties-of-extrapolation-space}
    Let $Z$ and $X$ be ordered Banach spaces and assume that there exists a bijection $T\in \calL(Z,X)_+$. 
    \begin{enumerate}[\upshape (a)]
        \item 
        If every positive, increasing, and norm-bounded net in $X$ is norm-convergent, then the same is true in $Z$.
        
        \item 
        If the cone of $X$ is a face of its bidual wedge, then the cone of $Z$ is also a face of its bidual wedge.
        
        \item 
        If the cone of $X$ is normal, then so is the cone of $Z$.
     \end{enumerate}
\end{theorem}

\begin{proof}
    (a) 
    Let $(z_{\alpha})$ be an increasing norm-bounded net in $Z_+$. Then $(Tz_{\alpha})$ is increasing and norm-bounded in $X_+$. So, by assumption, there exists $x \in X$ such that $Tz_{\alpha}\to x$ in $X$. The bounded inverse  theorem now implies that $z_{\alpha}\to T^{-1}x$ in $Z$, as desired. 

    (b) 
    Let $k:X\to X''$ and $\widetilde k:Z\to Z''$ be the canonical embeddings. Let $z \in Z_+$ and $z'' \in Z''_+$ such that $0\le z''\le \widetilde k(z)$. We need to show that $z'' \in \widetilde k(Z)$. 
    Since $T$ is positive, so is the double dual $T'':Z''\to X''$. Therefore, $0\le T'' z'' \le T'' \widetilde k(z)=k(Tz)$. 
    As $k(X_+)$ is a face of $X''_+$, it follows that $T''z'' \in k(X)$.  
    Hence, $z'' \in (T'')^{-1}k(X)=\widetilde k(T^{-1}X)=\widetilde k(Z)$.

    (c) 
    First of all, recall that normality of the cone in an ordered Banach space is equivalent to every order interval being norm-bounded \cite[Theorem~2.40]{AliprantisTourky2007}. 
    Now, let $a, b \in Z$. 
    Then $T[a,b] \subseteq [Ta,Tb]$ due to the positivity of $T$ and the later set is norm-bounded because of the normality of the cone in $X$. 
    Hence, $[a,b] = T^{-1}T[a,b]$ is also norm-bounded since $T^{-1}$ is continuous.
\end{proof}

Regarding part~(b) of the previous theorem, we note in passing that if a cone is a face in a wedge, then it follows that the latter is also a cone.

For a bipositive operator $J:X\to Z$ between ordered Banach spaces, we give conditions in Theorem~\ref{thm:face-in-extrapolation-space} that ensure that $J(X_+)$ is a face of $Z_+$. 
For the proof we need the following extension result for linear functionals.

\begin{proposition}
    \label{prop:extending-a-functional}
    Let $X$ be a pre-ordered Banach space with a generating wedge and let $V\subseteq X$ be a vector subspace. Let $x'\in X'_+$ and $\varphi: V\to \bbR$ be linear. The following are equivalent.
    \begin{enumerate}[\upshape (i)]
        \item There exists $y'\in X'$ such that $0\le y'\le x'$ and $y'\restrict{V}=\varphi$,

        \item For all $v\in V$ and $w\in X_+$ the inequality $v\le w$ implies $\varphi(v)\le \duality{x'}{w}$.
    \end{enumerate}
\end{proposition}

\begin{proof}
    ``(i) $\Rightarrow$  (ii)'':
    For each $v\in V, w\in X_+$ with $v\le w$, we have
    \[
        \varphi(v) = \duality{y'}{v}
                    \le \duality{y'}{w}
                    \le \duality{x'}{w},
    \]
    due to the positivity of $y'$. This proves (ii).

    ``(ii) $\Rightarrow$  (i)'': 
    Firstly, note that the functional $p: X\to [0,\infty)$ given by $x\mapsto \inf\{\duality{x'}{w}: w\in X_+, w\ge x\}$ is well-defined since $X_+$ is generating and can easily checked to be sublinear. 
    Moreover, by~(ii) one has $\varphi(v)\le p(v)$ for all $v \in V$.
    The Hahn-Banach theorem thus yields a linear functional $y':X\to \bbR$ that satisfies $y'\restrict{V}=\varphi$ and $y'(x)\le p(x)$ for all $x\in X$.
    For each $x\in X_+$, observe that
    $
        y'(-x)\le p(-x)=0,
    $
    so $y'$ is positive. 
    In particular, $y'\in X'$ since $X_+$ is generating \cite[Theorem~2.8]{ArendtNittka2009}.
    Finally, 
    $
        \duality{y'}{x} \le p(x) \le \duality{x'}{x}
    $
   for all $x \in X_+$ proves that $y'\le x'$.
\end{proof}

\begin{theorem}
    \label{thm:face-in-extrapolation-space}
    Let $X$ be an ordered Banach space with normal cone,
    let $Z$ be a pre-ordered Banach space, and let $J \in \calL(X,Z)_+$ have dense range.
    Assume that $(R_n)$ is a sequence in $\calL(Z,X)_+$ such that $(JR_n)$ and $(R_nJ)$ both converge to the identity operator in the weak operator topology on $\calL(Z)$ and $\calL(X)$ respectively.
    
    If  $X_+$ is a face of $X''_+$, then $J(X_+)$ is a face of $Z_+$ (in particular, $Z_+$ is a cone).
\end{theorem}

We point out that the map $J$ in the theorem is automatically bipositive: if $Jx \ge 0$, then $x$ is the weak limit of $(R_n J x) \subseteq X_+$, hence positive. Moreover, Example~\ref{exa:not-a-face} shows that the assumption that $X_+$ is a face of $X''_+$ cannot be dropped in Theorem~\ref{thm:face-in-extrapolation-space}.

\begin{proof}[Proof of Theorem~\ref{thm:face-in-extrapolation-space}]
    We denote by $k:X\to X''$ and $\widetilde k:Z\to Z''$ the canonical embeddings. 
    Then the following diagram commutes:
    \begin{center}
        \begin{tikzcd}
            X      \arrow{rrr}{J} \arrow{d}{k} & & & Z \arrow{d}{\widetilde k} \\
            X''    \arrow{rrr}{J''}            & & & Z''
        \end{tikzcd}
    \end{center}
    Let $x\in X_+$ and $z\in Z_+$ be such that $0\le z\le Jx$; we need to find $y \in X_+$ such that $Jy = z$. 
    To this end, we first find a $y'' \in X''_+$ such that $J''y'' = \widetilde k(z)$. 
    
    For the first step, define $\varphi:J'(Z')\to \bbR$ as 
    $
        v' \mapsto \langle \widetilde k(z), (J')^{-1}v'\rangle.
    $
    This is well-defined since $J'$ is injective as $J$ is assumed to have dense range.
    We want to extend $\varphi$ to a functional $y'' \in X''_+$ by employing Proposition~\ref{prop:extending-a-functional}.
    So let $v'=J'z'\in J'(Z') \subseteq X'$ and fix $w' \in X'_+$ such that $w'\ge v'$. Then
    \[
        \varphi(v')  
        = 
        \duality{\widetilde k(z)}{z'}
        = 
        \duality{z'}{z}
        = 
        \lim_{n\to\infty} \duality{z'}{JR_nz}
        = 
        \lim_{n\to\infty} \duality{v'}{R_nz}.
    \]
    Positivity of $z$ and $w'$ thus implies that
    \[
        \varphi(v') 
        \le 
        \limsup_{n\to\infty} \duality{w'}{R_nz}
        \le 
        \limsup_{n\to\infty} \duality{w'}{R_nJx}
        = 
        \duality{w'}{x}
        = 
        \duality{k(x)}{w'}.
    \]
    Thus by Proposition~\ref{prop:extending-a-functional} -- which is applicable since $X'_+$ is generating as $X_+$ is normal \cite[Theorem~2.40]{AliprantisTourky2007} -- there exists $y''\in X''$ such that $0\le y''\le k(x)$ and $y''\restrict{J'(Z')}=\varphi$. Now, $\widetilde k(z)=J''y''$, because for each $z'\in Z'$, we have
    \[
        \duality{\widetilde k(z)}{z'} = \varphi(J'z') = \duality{y''}{J'z'}= \duality{J''y''}{z'}.
    \]

    Now, we show that $y''$ even stems from an element of $X$. Since $k(X_+)$ is a face of $X''_+$ and $0\le y''\le k(x)$, there exists $y \in X$ such that $y''=k(y)$. As $k$ is bipositive, it follows that $0\le y\le x$. 
    Moreover, as $\widetilde k(z) = J''k(y) = \widetilde k (Jy)$, we conclude $z = Jy$.
\end{proof}

If $X$ and $D$ are Banach spaces, then we know as a consequence of Hahn-Banach theorem that every $T \in \calL(X,D)$ satisfies the norm equality $\norm{T'}_{D' \to X'} = \norm{T}_{X \to D}$. 
The following theorem can be interpreted as a variation of the inequality ``$\le$'' in this equality. 
For an intuition about the theorem, we refer to the subsequent Examples~\ref{exa:dual-estimate-on-cone-only}.

\begin{theorem}
    \label{thm:dual-estimate-on-cone-only}
    Let $X,D,E$, and $\tilde E$ be ordered Banach spaces such that the cone $\tilde E_+$ is generating and normal. 
    Consider positive bounded linear operators 
    \begin{center}
        \begin{tikzcd}
                           &                                            & E        \\[-15pt]
            X \arrow{r}{T} & D \arrow{ru}{J} \arrow[swap]{rd}{\tilde J} &          \\[-15pt]
                           &                                            & \tilde E 
        \end{tikzcd}
    \end{center}
    such that $\tilde J$ is bipositive and its range $\Ima \tilde J$ is majorizing in $\tilde E$.
    Then
    \[
        \norm{(JT)'}_{\linSpan{E'_+} \to X'} 
        \le 
        c \; 
        \lVert \tilde J T \rVert_{X \to \tilde E}
    \]
    for a number $c \ge 0$ that might depend on all involved spaces and operators except for $X$ and $T$. 
    Here, $\linSpan E'_+ = E'_+ - E'_+$ is endowed with the norm $\norm{\argument}_{E'_+-E'_+}$.
\end{theorem}

Note that the space $D$ does not occur in the conclusion of the theorem. 
We are interested in the mappings $\tilde J T: X \to \tilde E$ and $(JT)': E' \to X'$ and $D$ is an auxiliary space to relate their properties. Before giving the proof of the theorem, we illustrate two situations where the assumptions of the theorem are fulfilled:

\begin{examples}
    \label{exa:dual-estimate-on-cone-only}
    (a)
    Let $p \in [1,\infty]$, let $T\in \calL( L^p(\bbR) , L^\infty(\bbR))_+$, and assume that the range of $T$ lies in   
    $\operatorname{Lip}(\bbR)$, the space of all scalar-valued Lipschitz continuous functions on $\bbR$, which is an ordered Banach space when endowed with the pointwise order and the norm $\norm{f}_{\operatorname{Lip}} := \modulus{f(0)} + \operatorname{Lip}(f)$, where $\operatorname{Lip}(f)$ is the Lipschitz constant of $f$.
    Then Theorem~\ref{thm:dual-estimate-on-cone-only} implies that there exists $c>0$, independent of $T$ such that
    \[
        \norm{T'}_{\linSpan(\operatorname{Lip}(\bbR)'_+) \to (L^p)'} \le c \norm{T}_{L^p \to L^\infty}.
    \]

    To see this, let $X = L^p(\bbR)$, $E = \operatorname{Lip}(\bbR)$, $\tilde E = L^\infty(\bbR)$ and $D = E \cap \tilde E$ with the norm $\norm{f}_D := \norm{f}_{\operatorname{Lip}} + \norm{f}_\infty$ in the theorem, let $J$ and $\tilde J$ be the canonical embeddings, and note that $T$ is bounded from $L^p(\bbR)$ to $D$ by the closed graph theorem. Moreover, $D$ is majorizing in $\tilde E$ as the former contains the constant $\one$ function.
    
    (b)
    Let $\Omega$ denote the open unit ball in $\bbR^d$ and the consider Banach lattices $X=L^2(\Omega)$ and $\widetilde E=L^\infty(\Omega)$. Let $E$ and $D$ be ordered Banach spaces $H^4(\Omega) \cap H_0^2(\Omega)$ and $E\cap \widetilde E$ respectively, where the latter is equipped with the norm $\norm{\argument}_D:= \norm{\argument}_E + \norm{\argument}_{\infty}$; note that the inclusion $E \subseteq \tilde E$ holds if $d<8$, but not otherwise.
    
    The bi-Laplace operator $A:= -\Delta^2$ with domain $E$ generates a  \emph{uniformly eventually positive} analytic $C_0$-semigroup $(T(t))_{t\ge 0}$ on $X$ by \cite[Theorem~4.4]{DanersGlueck2018}, i.e., there exists $t_0 \ge 0$ such that $T(t)\ge 0$ for all $t\ge t_0$. Furthermore, $\dom A^n \subseteq \widetilde E$ for large $n \in \bbN$, so $T(t)X \subseteq D$ for all $t>0$. As in~(a), it can be verified that the assumptions of  Theorem~\ref{thm:dual-estimate-on-cone-only} are fulfilled and whence there exists $c> 0$ such that for all $t\ge t_0$,
    \[
        \norm{ T(t)'}_{\linSpan{E'_+} \to L^2} \le c \norm{ T(t) }_{L^2 \to L^\infty}.
    \]
     
\end{examples}

\begin{proof}[Proof of Theorem~\ref{thm:dual-estimate-on-cone-only}]
    As $\tilde E_+$ is generating in $\tilde E$ and $\tilde J(D)$ is majorizing in $\tilde E$, one can apply Proposition~\ref{prop:span-determined}(a) and~(c) to get $\tilde J'(\tilde E'_+) = D'_+$. 
    On the other hand, the dual wedge $\tilde E'_+$ is also generating owing to the normality of the cone $\tilde E_+$ \cite[Theorem~2.26]{AliprantisTourky2007} and thus $\tilde J'\tilde E' = D'_+ - D'_+ = \linSpan D'_+$. 
    Note that $\tilde J'$ is continuous from $\tilde E'$ to $\linSpan D'_+$ due to the closed graph theorem; here $\linSpan D'_+$ is endowed with the norm $\norm{\argument}_{D'_+-D'_+}$.
    Hence, the open mapping theorem implies that there exists a set $S \subseteq \tilde E'$ that is norm-bounded by some $\tilde c \ge 0$  such that $\tilde J'S$ covers the unit ball in $\linSpan D'_+$. 
    So,
    \[
        \norm{T'}_{\linSpan D'_+ \to X'} 
        \le 
        \tilde c \norm{T'\tilde J'}_{\tilde E' \to X'}
        =
        \tilde c \norm{\tilde J T}_{X \to \tilde E}
        .
    \]
    On the other hand, as $J'$ is positive it maps $\linSpan E'_+$ into $\linSpan D'_+$ 
    and it is a continuous operator between those Banach spaces by the closed graph theorem.
    Thus,
    \[
        \norm{(JT)'}_{\linSpan E'_+ \to X'} 
        \le 
        \norm{T'}_{\linSpan D'_+ \to X'} \norm{J'}_{\linSpan E'_+ \to \linSpan D'_+},
    \]
    which completes the proof when combined with the previous estimate.
\end{proof}

\subsection{The order on $X_{-1}$}
    \label{sec:order-properties-corollaries}

Let $(T(t))_{t\ge 0}$ be a positive $C_0$-semigroup on an ordered Banach space $X$. 
Throughout we endow the extrapolation space $X_{-1}$ with the cone 
\[
    X_{-1,+} := \overline{X_+}^{\norm{\argument}_{-1}},
\]
i.e., the norm closure of $X_+$ in $X_{-1}$, 
and consider the partial order that is induced by $X_{-1,+}$.
This order on $X_{-1}$ was, for the case where $X$ is a Banach lattice, introduced in \cite{BatkaiJacobVoigtWintermayr2018} in the context of perturbation theorems for positive semigroups. 
A number of fundamental properties of $X_{-1,+}$ -- in particular that it is indeed a cone -- are proved in \cite[Remark~2.2 and Proposition~2.3]{BatkaiJacobVoigtWintermayr2018}. 
We now show that the same properties remain true on ordered Banach spaces. 
Most arguments are similar, but we include the details for the convenience of the reader. 
Our proof of the equality $X_{-1,+}\cap (-X_{-1,+})=\{0\}$ is a bit different since the cone $X_+$ need not be normal in the following proposition.

\begin{proposition}
    \label{prop:order:extrapolation}
    Let $X$ be an ordered Banach space and let $(T(t))_{t\ge 0}$ be a positive $C_0$-semigroup on $X$.
    \begin{enumerate}[\upshape (a)]
        \item 
        The set $X_{-1,+}$ is a cone and hence, $(X_{-1},X_{-1,+})$ is an ordered Banach space.

        \item 
        The canonical embedding $X \hookrightarrow X_{-1}$ is bipositive, i.e., $X_+ = X_{-1,+} \cap X$.

        \item 
        For any $\lambda > \spb(A)$, the resolvent $\Res(\lambda, A_{-1})$ is positive from $X_{-1}$ to $X$.
    \end{enumerate}
\end{proposition}

\begin{proof}
    (c) 
    Let $\lambda > \spb(A) = \spb(A_{-1})$. 
    We need to prove that $\Res(\lambda,A_{-1})X_{-1,+} \subseteq X_+$.
    So let $x \in X_{-1,+}$ and choose a sequence $(x_n) \subseteq X_+$ that converges to $x$ in $X_{-1}$. Firstly, if $\lambda>\gbd(A)$, then the operator $\Res(\lambda,A)$ is positive from $X$ to $X$
    owing to  the Laplace transform representation of the resolvent. Now, the positivity extends from the interval $(\gbd(A), \infty)$ to the interval $(\spb(A), \infty)$ via the Taylor expansion of the resolvent and a connectedness argument; see for instance \cite[Proposition~2.1(a)]{GlueckMironchenko2025} for details.
    Therefore, $\Res(\lambda,A)x_n \in X_+$ for each index $n$. 
    As $\Res(\lambda,A_{-1})$ is continuous from $X_{-1}$ to $X$, it follows that
    \[
        \Res(\lambda,A_{-1}) x 
        = 
        \lim_{n \to \infty} \Res(\lambda,A_{-1}) x_n 
        = 
        \lim_{n \to \infty} \Res(\lambda,A) x_n 
        \in 
        X_+
        ,
    \]
    where both limits are taken in the Banach space $X$. 
    
    (a)
    It is easy to see that $X_{-1,+}$ is closed, convex, and that $\alpha X_{-1,+} \subseteq X_{-1,+}$ for each $\alpha \in [0,\infty)$, 
    so one only has to show that $X_{-1,+} \cap -X_{-1,+} = \{0\}$.
    To this end, suppose $x \in X_{-1}$ satisfies $\pm x \ge 0$. 
    Fixing $\lambda>\spb(A)$,
    it follows from~(c) that $\pm \Res(\lambda,A_{-1})x \in X_+$. 
    Since $X_+$ is a cone, this implies that $\Res(\lambda,A_{-1})x = 0$, 
    and so $x=0$ by injectivity of the resolvent operator.
    
    (b)
    Obviously, $X_+ \subseteq X_{-1,+} \cap X$.
    To establish the converse inclusion, we let $x \in X_{-1,+} \cap X$.
    It follows from~(c) that $\lambda \Res(\lambda,A)x = \lambda \Res(\lambda,A_{-1})x \in X_+$  for all $\lambda > \max\{\spb(A),0\}$. Moreover, since $x$ is in $X$ one has $\lambda \Res(\lambda,A)x \to x$ in $X$ as $\lambda \to \infty$. 
    As $X_+$ is closed in $X$, we even get that $x\in X_+$.
\end{proof}

Clearly, the extrapolated semigroup $(T_{-1}(t))_{t\ge 0}$ on $X_{-1}$ leaves $X_{-1,+}$ invariant.
Since $X_{-1}$ is, by the previous proposition, an ordered Banach space with respect to this cone, we can rephrase this by saying that the extrapolated semigroup is  positive.

Even if $X$ is a Banach lattice, the cone $X_{-1,+}$ need not be generating in $X_{-1}$ \cite[Examples~5.1 and~5.3]{BatkaiJacobVoigtWintermayr2018}. In particular, $X_{-1}$ is usually not a Banach lattice. 
One can show that the span of $X_{-1,+}$ is often a Banach lattice, though \cite[Section~4]{AroraGlueckSchwenninger2025}, but we will not use the observation in what follows. 
Nevertheless, for any $\lambda>\spb(A)$, we know from Proposition~\ref{prop:order:extrapolation} that the bijection $\Res(\lambda, A_{-1})\in \calL(X_{-1}, X)_+$. Hence we can apply the results proved in Section~\ref{sec:order-properties} to $Z=X_{-1}$ to obtain various order properties of $X_{-1}$.
Firstly, Theorem~\ref{thm:inherited-properties-of-extrapolation-space} readily implies the following result (cf.\ Table~\ref{table:properties}). Thereafter, we obtain sufficient conditions for the cone $X_+$ to be a face of the cone $X_{-1,+}$.

\begin{corollary}
    \label{cor:inherited-properties-of-extrapolation-space}
    Let $(T(t))_{t\ge 0}$ be a positive $C_0$-semigroup on an ordered Banach space $X$.
    \begin{enumerate}[\upshape (a)]
        \item 
        If every positive, increasing, and norm-bounded net in $X$ is norm-convergent, then the same is true in $X_{-1}$.

        \item
        If the cone of $X$ is a face of the bidual wedge, then the same is true for the cone of $X_{-1}$.

        \item 
        If the cone of $X$ is normal, then so is the cone of $X_{-1}$.
     \end{enumerate}
\end{corollary}

\begin{corollary}
    \label{cor:face-in-extrapolation-space}
        Let $X$ be an ordered Banach space with a generating and normal cone and assume that $X_+$ is a face of $X''_+$.
        If $(T(t))_{t\ge 0}$ is a positive $C_0$-semigroup on $X$, then $X_+$ is a face of $X_{-1,+}$.
\end{corollary}

\begin{proof}
    By definition, $X$ is dense in $X_{-1}$ and according to Proposition~\ref{prop:order:extrapolation}(b) the canonical embedding is bipositive. Therefore, defining $T=\Res(\lambda,A_{-1})$ for fixed $\lambda>\spb(A)$ and $R_n:=n\Res(n, A_{-1})$ for sufficiently large $n\in \bbN$, Theorem~\ref{thm:face-in-extrapolation-space} implies that $X_+$ is indeed a face of $X_{-1,+}$.
\end{proof}

The following example shows that Corollary~\ref{cor:face-in-extrapolation-space} fails if $X_+$ is not a face of $X''_+$.

\begin{example}
    \label{exa:not-a-face}
    On the space $X=\{f \in C[0,1]:f(0)=f(1)\}$,
    the operator
    \[
        \dom A :=\{ f \in C^1[0,1] \in X:f'(0)=f'(1)\},\qquad
              f  \mapsto f' 
    \]
    generates a positive periodic shift semigroup \cite[Section~A-I.2.5]{Nagel1986} for which the extrapolation space was shown in \cite[Example~5.3]{BatkaiJacobVoigtWintermayr2018} to be 
    \[
        X_{-1} = \{ g \in \calD(0,1)' : g = f - \partial f \text{ for some }f \in X\};
    \]
    here $\calD(0,1)$ is the space of test functions on $(0,1)$ and $\partial$ is the distributional derivative. The function
    \[
        f(x) = 
                \begin{cases}
                    1 - \frac{e^x}{1+\sqrt{e}} \qquad & \text{ for }x \in [0,\frac12)\\
                    \frac{e^x}{e+\sqrt{e}}     \qquad & \text{ for }x \in [\frac12,1]
                \end{cases}
    \]
    lies in $X$ and satisfies $f-\partial f =\one_{[0,\frac12]}$. In particular,  $\one_{[0,\frac12]} \in X_{-1}$ and  $0 \le \one_{[0,\frac12]} \le \one$. As $X_+$ contains $\one $ but not $\one_{[0,\frac12]}$, we conclude that $X_{+}$ is not a face of $X_{-1,+}$.
\end{example}

We now present an application of Theorem~\ref{thm:dual-estimate-on-cone-only} and follow it up with two examples involving elliptic operators. 
To state the result, let us first recall the definition of Favard spaces with index $\beta \in (-1,0]$. 
Let $(T(t))_{t \ge 0}$ be a $C_0$-semigroup on a Banach space $X$ with generator $A$ and fix  $\omega>\gbd(A)$.
For each $x \in X_{-1}$, one defines 
\begin{align*}
    \norm{x}_\beta 
    := 
    \sup_{t \in (0,\infty)} \frac{1}{t^{\beta+1}} \norm{e^{-\omega t} T_{-1}(t) x - x}_{-1} \in [0,\infty]
    .
\end{align*}
The space $F_\beta := \{x \in X_{-1} : \, \norm{x}_\beta < \infty\}$ is called the \emph{Favard space of order $\beta$} of the semigroup. 
It is a Banach space with respect to the norm $\norm{\argument}_\beta$; 
see, for instance, \cite[Definitions~II.5.10 and~II.5.11]{EngelNagel2000} for details. 

\begin{theorem}
    \label{thm:ultracontractive}
    Let $(T(t))_{t\ge 0}$ be a positive and immediately differentiable $C_0$-semigroup on a reflexive ordered Banach space $X$ whose cone $X_+$ is generating. 
    Let $\varphi \in X'$ be strictly positive.
    Assume that $T(t)' X' \subseteq (X')_\varphi$ for all $t > 0$ and $\dom A' \cap (X')_\varphi$ is majorizing in $(X')_\varphi$.
    \begin{enumerate}[\upshape (a)]
        \item 
        There exists a number $c \ge 0$ such that 
        \[
            \norm{T_{-1}(t)}_{\linSpan X_{-1,+} \to X} 
            \le 
            c \norm{T(t)'}_{X' \to (X')_{\varphi}} \qquad (t>0).
        \]

        \item 
        If $(T(t))_{t\ge 0}$ is analytic and the dual semigroup satisfies an ultracontractivity type estimate $\norm{T(t)'}_{X' \to (X')_{\varphi}} \le \widetilde c t^{-\alpha}$ for some $\widetilde c \ge 0$, $\alpha \in [0,1)$, and all $t \in (0,1]$, then $\linSpan X_{-1,+}$ is contained in the Favard space $F_{-\alpha}$.
    \end{enumerate}
\end{theorem}

For the definition of $(X')_\varphi$, we refer to~\eqref{eq:principal-ideal}. 
We point out that under the assumptions of Theorem~\ref{thm:ultracontractive}, the extrapolation semigroup $(T_{-1}(t))_{t\ge 0}$ is also immediately differentiable and hence maps $X_{-1}$ to $\dom A_{-1}=X$ for all $t>0$. 

\begin{proof}[Proof of Theorem~\ref{thm:ultracontractive}]
    (a) 
    We apply Theorem~\ref{thm:dual-estimate-on-cone-only} to the following situation: 
    choose the space $X$ from Theorem~\ref{thm:dual-estimate-on-cone-only} as the space $X'$ in the present theorem, 
    let $E := \dom A'$ and $\widetilde E := (X')_\varphi$, 
    Note that as the cone in $X$ is generating, the cone in $X'$ is normal \cite[Theorem~2.42]{AliprantisTourky2007} and hence $\widetilde E = (X')_\varphi$ is an ordered Banach space with a normal cone with respect to the gauge norm $\norm{\argument}_\varphi$ \cite[Theorems~2.60 and~2.63]{AliprantisTourky2007}.
    Moreover, the cone in $\widetilde E$ can easily be checked to be generating.
    Finally, choose $D := E \cap \widetilde E$ which is an ordered Banach space with respect to the norm $\norm{\argument}_D := \norm{\argument}_E + \norm{\argument}_{\widetilde E}$. 
    Let $J$ and $\widetilde J$ be the canonical embeddings of $D$ into $E$ and $\widetilde E$ and note that the range of $\widetilde J$ is majorizing in $\widetilde E$ by the assumptions of the present theorem. 
    Fix $t >0$ and let $T:=T(t)'$ -- our assumptions imply that $\Ima T \subseteq D$.
    Hence, Theorem~\ref{thm:dual-estimate-on-cone-only} yields a $c>0$, independent of $t$, such that
    \[
        \norm{(T(t)')'}_{\linSpan({\dom(A')'_+}) \to X''} 
        \le 
        c \; 
        \lVert \widetilde T(t)' \rVert_{X' \to (X')_\varphi}
        .
    \]
    Due to the reflexivity, $X_{-1}$ can be identified with $\dom(A')'$ as an ordered Banach space. Therefore, one can replace the left-hand side of the previous estimate with $\norm{T_{-1}(t)}_{\linSpan X_{-1,+} \to X}$ and doing so yields the claim.
    
    (b) 
    This is a consequence of~(a) and a characterisation \cite[Proposition~II.5.13 and Definition~II.5.11]{EngelNagel2000} of Favard spaces for analytic semigroups.
\end{proof}

\begin{example}[$X_{-1}$ for elliptic operators with Neumann boundary conditions]
    \label{exa:X-1-for-neumann}
    Let $\Omega \subseteq \bbR^d$ be a bounded domain with Lipschitz boundary and let $p \in (1,\infty)$. 
    Let $a \in L^\infty(\Omega; \bbR^{d \times d})$ satisfy the following coercivity condition: 
    there exists a number $\nu > 0$ such that for almost all $x \in \Omega$ the estimate
    \[
        \xi^T a(x) \overline{\xi} \ge \nu \norm{\xi}^2 
        \qquad \text{for all } \xi \in \bbC^d
    \]
    holds.
    Consider the divergence form elliptic operator $A: L^p(\Omega) \supseteq \dom A \to L^p(\Omega)$, $u \mapsto \div(a \nabla u)$ with Neumann boundary conditions. 
    This operator can be constructed using form methods in $L^2(\Omega)$ and then extrapolating to the $L^p$-scale, see \cite[Section~4.1]{Ouhabaz2005}. 
    The operator $A$ generates a positive analytic $C_0$-semigroup $(T(t))_{t \ge 0}$ on $L^p(\Omega)$ \cite[Corollary~4.3 and Theorem~1.52]{Ouhabaz2005}. 
    Let $q \in (1,\infty)$ be the Hölder conjugate of $p$ and suppose $\alpha := \frac{d}{2q} <1$. Then the following hold.
    \begin{enumerate}[(a)]
        \item 
        \emph{The span of the cone $L^p(\Omega)_{-1,+}$ is contained in the Favard space $F_{-\alpha}$}. 
        Indeed, the dual operator $A'$ on the Banach space $L^{q}(\Omega)$ is also a divergence form elliptic operator with Neumann boundary conditions, but with the diffusion coefficient $a^T$. 
        Therefore the semigroup $(T'(t))_{t \ge 0}$ generated by $A'$ maps $L^{q}$ into $L^\infty(\Omega) = \big(L^{q}(\Omega)\big)_{\one}$ and satisfies the ultracontractivity estimate
        \[
            \norm{T(t)'}_{L^{q} \to L^\infty} 
            \le 
            c t^{-\frac{d}{2}\frac{1}{q}} 
            =
            c t^{-\alpha}
        \]
        for some $c > 0$ and all $t \in (0,1]$;
        see \cite[Sections~7.3.2 and~7.3.6]{Arendt2004}.
        Moreover, $\dom A'$ contains the constant function $\one$,
        so $\dom A' \cap L^\infty(\Omega)$ is majorizing in $L^\infty(\Omega)$. 
        Hence, the assumptions of Theorem~\ref{thm:ultracontractive}(b) are satisfied.

        \item 
        \emph{The cone $L^p(\Omega)_{-1,+}$ is the set of all positive finite Borel measures on $\overline{\Omega}$}. 
        Indeed, as $L^p(\Omega)$ is reflexive, the space $L^p(\Omega)_{-1}$ can be identified with $\dom(A')'$ as an ordered Banach space. 
        So we only need to identify the positive linear functionals on $\dom A'$.
        Since $\alpha < 1$ it follows that $\dom A'$ is contained in the space $C(\overline{\Omega})$ \cite[Proposition~3.6]{Nittka2011}. 
        Moreover, $\dom A'$ is majorizing in $C(\overline{\Omega})$ since it contains the constant functions and is dense in $C(\overline{\Omega})$ by \cite[Lemma~4.2]{Nittka2011}. 
        Since the cone in $C(\overline{\Omega})$ has non-empty interior it follows that even the cone $(\dom A')_+$ is dense in the cone $C(\overline{\Omega})_+$, so by Proposition~\ref{prop:span-determined}(b) the dual space $C(\overline{\Omega})'$ embeds bipositively into the space $\dom(A')' \simeq L^p(\Omega)_{-1}$ and by part~(c) of the same proposition, this embedding maps $C(\overline{\Omega})'_+$ surjectively onto the cone of $\dom(A')' \simeq L^p(\Omega)_{-1}$. 
        As the positive cone in $C(\overline{\Omega})'$ consists of all positive finite Borel measures on $\overline{\Omega}$, the claim follows.
    \end{enumerate}
\end{example}

For elliptic operators with Dirichlet boundary conditions, the situation is more subtle: 
in order to apply Theorem~\ref{thm:ultracontractive} one needs the semigroup to be \emph{intrinsically ultracontractive}, meaning that it not only maps $L^1$ into $L^\infty$ but also maps a weighted $L^1$-space -- with the leading eigenfuntion $u$ of the operator as a weight -- into the principal ideal generated by $u$. 
Such results exist in the literature in the case that the domain and the coefficients of the differential operator are sufficiently smooth, see for instance, \cite[Theorem~4.6.2]{Davies1989}. 
However, to get an estimate from $L^p$ into the principal ideal generated by $u$ we also need $p$ to be at least $2$ -- this is why the result in the following example can only be applied if the spatial domain is an interval.

\begin{example}[$X_{-1}$ for elliptic operators with Dirichlet boundary conditions on intervals]
    \label{exa:X-1-for-dirichlet} 
    Let $p \in [2,\infty)$, let $\Omega \subseteq \bbR$ be a bounded open interval, and let $A: L^p(\Omega) \supseteq \dom A \to L^p(\Omega)$ denote the same operator as in Example~\ref{exa:X-1-for-neumann}, but now with Dirichlet boundary conditions. 
    Assume in addition that the coefficient $a$ is $C^1$ on $\overline{\Omega}$.
    Again, it is well-known that $A$ generates a positive and analytic $C_0$-semigroup $(T(t))_{t \ge 0}$ on $L^p(\Omega)$. 
    Let $q \in (1,2]$ be the Hölder conjugate of $p$ and suppose $\alpha := \frac{3}{2q} < 1$ 
    (all of the following arguments would also work on smooth $d$-dimensional domains if the coefficient $a$ is symmetric and  $\alpha := \frac{1}{q} (1+\frac{d}{2}) < 1$; 
    but this inequality in conjunction with $q \le 2$ is possible only if $d = 1$).
    
    \emph{We show that $L^p(\Omega)_{-1,+}$ is contained in the Favard space $F_{-\alpha}$}.
    We begin by noting that since the coefficient matrix $a$ is assumed to be symmetric, the dual operator $T(t)'$ acts as the continuous extension of $T(t)$ to  $L^{q}(\Omega)$. 
    Let $u \in L^p(\Omega)_+$ denote the leading eigenfunction of $A$ and thus also of $A'$ (in particular, $\varphi :=u \in L^{q}(\Omega)_+$) associated to the leading eigenvalue $-\lambda_0 \in (-\infty,0)$.
    We show that the assumptions of Theorem~\ref{thm:ultracontractive}(a) are satisfied.
    
    The smoothness assumption on the coefficient $a$ implies that the heat kernel $k_t: \Omega \times \Omega \to [0,\infty)$ of the semigroup operator $T(t)$ (and thus also of $T(t)'$) satisfies the intrinsic ultracontractivity estimate
    \begin{equation}
        \label{eq:exa:X-1-for-dirichlet-kernel-estimate}
        k_t(x,y) \le c t^{-\frac{3}{2}} u(x) u(y)
        = 
        c t^{-\alpha q} u(x) u(y)
    \end{equation}
    for a constant $c > 0$, all $t > 0$, and all $x,y \in \Omega$ \cite[Theorem~4.6.2]{Davies1989}. 
    Hence, $T(t)'$ maps each $L^{q}(\Omega)$ into the principal ideal $(L^{q}(\Omega))_u$. 
    Moreover, since $u \in \dom A'$, the space $\dom A' \cap (L^{q}(\Omega))_u$ is clearly majorizing in $(L^{q}(\Omega))_u$.
    So it only remains to show that $\norm{T(t)'}_{L^{q} \to (L^{q})_u} \le \tilde c t^{-\alpha}$ for a constant $\tilde c > 0$ and all $t \in (0,1]$.

    To this end,  we first observe that the semigroup $(T(t))_{t\geq 0}$ is dominated by the semigroup generated by the same differential operator but with Neumann boundary conditions; this follows for instance from \cite[Corollary~2.22]{Ouhabaz2005}. 
    In addition to the principal ideals $(L^p(\Omega))_u$ and $(L^{q}(\Omega))_u$ we will now use the weighted $L^1$-space $L^1(\Omega, u \dxMedium x)$. 
    This space coincides with the norm completion of $L^p(\Omega)$ with respect to the norm $\norm{f}_{L^1(\dxShort x)} := \int_\Omega \modulus{f} u \dx x$ and behaves, in a sense, dually to the principal ideal $(L^{q}(\Omega))_u$; we refer to \cite[Section~2]{AroraGlueck2022a} or \cite[Section~2]{DanersGlueck2018} for a detailed discussion. 
    
    Moreover, since $T(t)'u = e^{-t\lambda_0} u$ for each $t$, the operator $T(t)$ extends to a bounded linear operator on $L^1(\Omega, u \dxMedium x)$, 
    again denoted by $T(t)$, which satisfies
    \[
        \norm{T(t)}_{L^1(u\dxMedium x) \to L^1(u\dxMedium x)} 
        \le 
        e^{-t\lambda_0}
        \le 
        1
        .
    \]
    At the same time it follows from the heat kernel estimate~\eqref{eq:exa:X-1-for-dirichlet-kernel-estimate} that, for each $t > 0$, the operator $T(t)$ maps $L^1(\Omega, u\dxMedium x)$ into the principal ideal $(L^p(\Omega))_u$ with norm 
    \[
        \norm{T(t)}_{L^1(u\dxMedium x) \to (L^p)_u} 
        \le ct^{-\alpha q}
    \]
    (an abstract version of this argument can be found in \cite[Proposition~2.2]{AroraGlueck2022a}).
    We now combine the previous two norm estimates with an interpolation inequality: 
    for every function $f \in (L^p(\Omega))_u \subseteq L^p(\Omega) \subseteq L^1(\Omega, u \dxMedium x)$, we have
    \[
        \norm{f}_{L^p}^p 
        = 
        \int_\Omega u^{p-2} \; \modulus{f} u \; \frac{\modulus{f}^{p-1}}{u^{p-1}} \dx x
        \le 
        \norm{u^{p-2}}_\infty \norm{f}_{L^1(u\dxMedium x)} \norm{f}_{(L^p)_u}^{p-1},
    \]
    so $\norm{f}_{L^p} \le \norm{u^{p-2}}_\infty^{1/p} \norm{f}_{L^1(u\dx x)}^{1/p} \norm{f}_{(L^p)_u}^{1/q}$;
    here we have also used that $u \in L^\infty(\Omega)$ and $p \ge 2$, which ensures $u^{p-2} \in L^\infty(\Omega)$.
    This estimate for $\norm{f}_{L^p}$ together with the two aforementioned norm estimates for $T(t)$ readily yield
    \[
        \norm{T(t)}_{L^1(u\dxMedium x) \to L^p} 
        \le 
         c^{1/q}\norm{u^{p-2}}_\infty^{1/p} t^{-\alpha}
    \]
    for all $t > 0$. 
    For every $t > 0$, one has $\norm{T(t)}_{L^1(u\dxMedium x) \to L^p} = \norm{T(t)'}_{L^{q} \to (L^{q})_u}$.
    Indeed, the inequality $\le$ is shown in an abstract setting in \cite[Proposition~2.1]{DanersGlueck2018}, 
    but one can show that even equality is true by using the duality result in \cite[Exercise~IV.9(a)]{Schaefer1974}.
    So the assumption of Theorem~\ref{thm:ultracontractive}(b) is satisfied, which gives the claimed result.
\end{example}

For Dirichlet boundary conditions, not all elements of $L^p(\Omega)_{-1,+}$ can be described as measures on $\overline{\Omega}$. This is in contrast to the situation for Neumann boundary conditions and small $p$, see Example~\ref{exa:X-1-for-neumann}(b).
One can see this by taking $A$ to be the Dirichlet Laplacian on the interval $(-1,1)$. 
Then $A'$ is also the Dirichlet Laplacian and the domain $\dom A'$ is (no matter which value we use for $p$) a majorizing and dense subset of $C^1[-1,1]$.
Hence, $f\mapsto \mp f'(\pm1)$ are positive elements of $\dom(A')' \simeq L^p((-1,1))_{-1}$ that are not given by measures on $[-1,1]$.

\section{Admissibility of observation operators}
    \label{sec:admissibility-observation}

We now come to the main topic of the paper -- admissibility.
In this section, we are interested in the system
\begin{equation*}
    \label{eq:system-observation}
    \Sigma(A,C)\quad\left\{
        \begin{aligned}
            \dot{x}(t) &= Ax(t),\quad &t\ge 0\\
            y(t)       &= Cx(t),       \quad &t\ge 0\\
            x(0)       &= x_0
        \end{aligned}
    \right.;
\end{equation*}
where $A$ generates a $C_0$-semigroup $(T(t))_{t\ge 0}$ on a Banach space $X$ and the \emph{observation} operator $C\in\calL(X_1,Y)$ takes values in a Banach space $Y$. 
We say that $C$ is a \emph{$\mathrm L^1$-admissible observation operator} if the \emph{output map} -- defined as
\begin{equation}
    \label{eq:output-operator}
        \Psi_{\tau}:X_1 \to     L^1([0,\tau],Y),\qquad
                    x   \mapsto   CT(\argument)x
\end{equation}
has a bounded extension to $X$ for some (equivalently, all) $\tau>0$. This is of course, equivalent to the existence of $K_{\tau}>0$ such that 
    \begin{equation}
        \label{eq:observation-equivalent}
        \norm{CT(\argument)x}_{L^1([0,\tau],Y)} \le K_{\tau}\norm{x}\qquad (x\in X_1).
    \end{equation}
In addition, if we have $\limsup_{\tau\downarrow 0} \norm{\Psi_{\tau}}_{\calL(X,\mathrm{L^1}([0,\tau],Y))}=0$, then $C$ is called a \emph{zero-class $\mathrm{L^1}$-admissible observation operator}. 
Most of our results in this section are proved in the following setting:
\begin{assumption}
    \label{ass:general}
    Suppose that $X$ is an ordered Banach space with a generating and normal cone and $(T(t))_{t\ge 0}$ is a positive $C_0$-semigroup on $X$. Moreover, $Y$ is a Banach space and
    $C\in \calL(X_1,Y)$.
\end{assumption}

In all of our results, the order on $X_1$ is assumed to be the order inherited from $X$. 
We begin with obtaining sufficient conditions for an operator $C \in \calL(X_1,Y)$ to be $L^1$-admissible. Thereafter, in Section~\ref{sec:zero-class-output}, we look at situations where $L^1$-admissibility of $C$ automatically implies zero-class $L^1$-admissibility.

\subsection{$L^1$-admissibility}

For positive semigroups on Banach lattices, it is known from \cite[Theorem~4.1.17]{Wintermayr2019} that positive observation operators mapping into AL-spaces (cf. Table~\ref{table:properties}) are always $L^1$-admissible. In our first result, we generalise this to ordered Banach spaces whose norm is additive on the positive cone. Consequently, we are able to show that every positive finite-rank observation operator is $L^1$-admissible (Corollary~\ref{cor:admissibility-observation-finite-rank}).

\begin{proposition}
    \label{prop:admissibility-observation-additive-norm}
    Suppose that Assumption~\ref{ass:general} holds. 
    Assume further that $Y$ is an ordered Banach space whose norm is additive on the positive cone and that $C\in \calL(X_1,Y)$ is positive.
    Then $C$ is $L^1$-admissible.
\end{proposition}

Examples of ordered Banach space whose norm is additive on the positive cone but are not AL-spaces are non-commutative $L^1$-spaces, pre-duals of non-commutative von Neumann algebras, or ice-cream cone in finite-dimensions. Some more examples are given in \cite[Appendix~A]{GlueckWolff2019}.

For the proof of Proposition~\ref{prop:admissibility-observation-additive-norm}, we essentially adapt the argument of \cite[Theorem~4.1.17]{Wintermayr2019} to our setting. The essential ingredient of the argument is that due to the positivity of the semigroup and the observation operator, it suffices to establish~\eqref{eq:observation-equivalent} for positive $x\in X_1$. When $X$ and $Y$ are both Banach lattices, this was proved in \cite[Lemma~3.1]{Gantouh2022a}. However, the result remains true under our assumption as well and is actually a consequence of the following technical lemma.

\begin{lemma}
    \label{lem:approximation-argument}
    Let $X, \widetilde X$ be ordered Banach spaces, let the cone $X_+$ be generating in $X$, and let $Y$ be a Banach space. 
    Let $(F(t))_{t \in [0,1]}$ be a family of operators in $\calL(\widetilde X,Y)$ with strongly measurable orbits.  
    Let $J \in \calL(\widetilde X, X)_+$ and assume the following:
    \begin{enumerate}[\upshape (a)]
        \item 
        There exists a sequence $(R_n) \subseteq \calL(X, \widetilde X)_+$ such that $(J R_n)$ and $(R_n J)$ converge strongly to the identity operators on $X$ and $\widetilde X$, respectively.
        
        \item 
        There exists $\widetilde M \ge 0$ such that 
        $
            \int_0^{1} \norm{F(t)x} \dx t \le \widetilde M \norm{Jx}
        $
        for each $x\in \widetilde{X}_+$.
    \end{enumerate}
    Then there exists $M \ge 0$ such that
    $
        \int_0^{1} \norm{F(t)x} \dx t \le  M \norm{Jx}
    $
    even for all $x\in \widetilde X$.
\end{lemma}

\begin{proof}
    Let $x \in \widetilde X$. 
    By the uniform decomposition property of ordered Banach spaces with generating cones \cite[Theorem~2.37]{AliprantisTourky2007}, there exists $c>0$ -- independent of $x$ -- and $y,z\in X_+$ such that $Jx=y-z$ and $\norm{y},\norm{z}\le c\norm{Jx}$.
    Due to~(b),
    \[
        \int_0^1 \norm{F(t) R_n Jx} \dx t 
        \le 
        \widetilde M \left(\norm{JR_n y}+\norm{JR_n z}\right)
    \]
    for all $n \in \bbN$.
    Fatou's lemma and the convergence properties in~(a) thus give
    \[
        \int_0^1 \norm{F(t) x} \dx t 
        \le 
        \widetilde M\big( \norm{y} + \norm{z} \big) 
        \le 
        2c \widetilde M \norm{Jx}
        ,
    \]
    as desired.
\end{proof}

\begin{proof}[Proof of Proposition~\ref{prop:admissibility-observation-additive-norm}]
    Let $C\in \calL(X_1,Y)$ be positive and assume without loss of generality that $\gbd(A)<0$. 
    Then for every positive $x \in X_1$, the integral $\int_0^1 \norm{CT(t)x} \dx t$ can be estimated from above by
    \[  
        \int_0^{\infty} \norm{CT(t)x} \dx t
        = \norm{\int_0^{\infty} CT(t)x \dx t}
        = \norm{CA^{-1}x}
        \le \norm{CA^{-1}} \norm{x};
    \]
    here the first equality holds because the norm on the positive cone of $Y$ is additive.
    Now we can apply Lemma~\ref{lem:approximation-argument} to the space $\tilde X := X_1$, 
    the canonical embedding $J \in \calL(X_1, X)_+$, the resolvent operators $R_n := n\Res(n,A) \in \calL(X,X_1)_+$, 
    and the operators $F(t) := CT(t) \in \calL(X_1,Y)$.
\end{proof}

The following sufficient criterion for $L^1$-admissibility of observation operators follows directly from Proposition~\ref{prop:admissibility-observation-additive-norm}.

\begin{theorem}
    \label{thm:admissibility-observation-factorisation}
    Suppose that Assumption~\ref{ass:general} holds.
    Assume further that there exists an ordered Banach space $\widetilde X$ whose norm is additive on the cone $\widetilde{X}_+$ such that $C$ factorises as
    \[
            C: X_1 \overset{C_1}{\longrightarrow} \widetilde X \overset{C_2}{\longrightarrow} Y
    \]
    for a positive operator $C_1 \in \calL(X_1,\widetilde X)$ and an operator $C_2 \in \calL(\widetilde X, Y)$.
    Then $C$ is $L^{1}$-admissible. 
\end{theorem}

\begin{proof}
    By Proposition~\ref{prop:admissibility-observation-additive-norm}, the operator
    $C_1$ is $L^{1}$-admissible. Thus, for $\tau > 0$ there exists $K_{\tau}>0$ such that
    \[
        \norm{C_2 C_1 T(t)x}_{L^1([0,\tau],Y)} \le  K_{\tau}\norm{C_2}\norm{x}
    \]
    for all $x\in X_1$, As a result, $C=C_2 C_1$ is also $L^{1}$-admissible.
\end{proof}

The factorisation condition in the preceding theorem might seem artificial at first glance. 
Yet, the next two results show how the condition arises naturally in some situations.

\begin{corollary}
    \label{cor:admissibility-observation-finite-rank}
    Suppose that Assumption~\ref{ass:general} holds. If $Y$ is an ordered Banach space and $C\in \calL(X_1,Y)$ is a positive finite-rank operator, then $C$ is $L^1$-admissible.
\end{corollary}

\begin{proof}
    It suffices to find an ordered Banach space $\widetilde X$  such that the assumptions of Theorem~\ref{thm:admissibility-observation-factorisation} hold.
    Define $\widetilde X \subseteq Y$ to be the range $\Ima C$ endowed with the cone $\widetilde{X}_+ := \widetilde X \cap Y_+$. 
    Then $\widetilde X$ is a finite-dimensional ordered Banach space and due to finite-dimensionality there exists an equivalent norm on this space that is additive on the positive cone; see \cite[Corollary~3.8]{AliprantisTourky2007} and \cite[Theorem~VII.1.1]{Wulich2017}. 
    Setting $C_1 := C: X_1\to \widetilde X$ and $C_2:\widetilde X\to Y$ to be the canonical embedding, Theorem~\ref{thm:admissibility-observation-factorisation} can be applied.
\end{proof}

We point out that in Corollary~\ref{cor:admissibility-observation-finite-rank}, although $\Ima C$ is a finite-dimensional ordered (Banach) space, it need not be a lattice. So, even though they're exists an equivalent norm on $\Ima C$ that is additive on the positive cone, there may not be an equivalent norm on $\Ima C$ that turns it into a lattice. For this reason, we needed the generalisation of \cite[Theorem~4.1.17]{Wintermayr2019} given in Proposition~\ref{prop:admissibility-observation-additive-norm}.

\begin{remark}
    In general, a non-positive finite-rank observation operator need not be $L^1$-admissible. Indeed, for the left translation semigroup on $X:=C_0[0,1)$, the observation operator $C:X_1\to \bbC$ given by $f\mapsto f'(0)$ is not $L^1$-admissible. 
    
    In fact, even analyticity of the semigroup is not sufficient to guarantee $L^1$-admissibility of a finite-rank operator \cite[Theorem~10]{JacobSchwenningerZwart2019}. 
\end{remark}

The dual of the factorization assumption can be characterised in terms of order boundedness (Appendix~\ref{sec:factorisation}). This lets us
reformulate Theorem~\ref{thm:admissibility-observation-factorisation} for reflexive $Y$:

\begin{corollary}
    \label{cor:admissibility-observation-order-bounded}
    Suppose that Assumption~\ref{ass:general} holds and assume that $Y$ is reflexive.    
    If the dual operator $C': Y'\to (X_1)'=(X')_{-1}$ maps the unit ball of $Y'$ into an order bounded subset of $(X')_{-1}$, then $C$ is $L^{1}$-admissible. 
\end{corollary}

\begin{proof}
    We know from Corollary~\ref{cor:inherited-properties-of-extrapolation-space}(c), that the cone of $(X')_{-1}$ is normal.
    Therefore, we may deduce from Proposition~\ref{prop:unit-ball-into-order-bounded} that there exists an ordered Banach space $\widetilde X$ with a unit such that $C'$ factorises as 
    \[
            C': Y' \overset{C_1}{\longrightarrow} \widetilde X \overset{C_2}{\longrightarrow} (X')_{-1};
    \]
    for an operator $C_1 \in \calL(Y',\widetilde X)$ and a positive operator $C_2 \in \calL(\widetilde X, (X')_{-1})$.
    Taking duals yields 
    \[
            C'': (X_1)'' \overset{C_2'}{\longrightarrow} (\widetilde X)' \overset{C_1'}{\longrightarrow} Y'';
    \]
    here we have used the equality $(X')_{-1}=(X_1)'$. Restricting ourselves to $X_1$ and employing the reflexivity of $Y$ gives the factorisation
    \[
            C: X_1 \overset{C_2'}{\longrightarrow} (\widetilde X)' \overset{C_1'}{\longrightarrow} Y.
    \]

    Finally, the norm on $(\widetilde X)'$ is additive on its positive cone because $\widetilde X$ has a unit \cite[Lemmata~2 and~3]{Ng1969}. Moreover, $C_2'$ is positive because $C_2$ is. In particular, all assumptions of Theorem~\ref{thm:admissibility-observation-factorisation} are fulfilled and thus $C$ is $L^{1}$-admissible. 
\end{proof}

\subsection{Zero-class $L^1$-admissibility}
    \label{sec:zero-class-output}

While $L^1$-admissible observation operators need not be zero-class $L^1$-admissible (Example~\ref{exa:shift-L1-endpoint}), we are interested in situations when this is indeed the case. To this end, we take advantage of the nilpotency of the right-translation semigroup on $L^1([0,1], Y)$ to formulate a simple test for $L^1$-admissibility to be zero-class.

\begin{proposition}
    \label{prop:zero-class-translation}
    Let $(T(t))_{t\ge 0}$ be a $C_0$-semigroup on a Banach space $X$, let $Y$ be a Banach space, and let $(R_t)_{t\ge 0}$ denote the nilpotent right-translation semigroup on $L^1([0,1], Y)$.
    
    Suppose that $C \in \calL(X_1,Y)$ is $L^1$-admissible.
    Then $C$ is zero-class $L^1$-admissible if and only if $\lim_{t\uparrow 1}\norm{R_{t}\Psi_1}_{L^1}=0$, where $\Psi_{(\argument)}$ is the output map defined in~\eqref{eq:output-operator}.
\end{proposition}

\begin{proof}
    For each $\tau<1$ and $x\in X_1$, we have
    \begin{align*}
        \norm{R_{1-\tau}\Psi_1x}_{L^1} & = \int_0^1 \norm{(R_{1-\tau}\Psi_1x)(s)}_Y \dx s
                                   = \int_{1-\tau}^1 \norm{(\Psi_1x)(s-1+\tau)}_Y \dx s\\
                                   & = \int_{1-\tau}^1 \norm{CT(s-1+\tau)x}_Y \dx s
                                   = \int_0^{\tau} \norm{CT(s)x}_Y\dx s
                                   = \norm{\Psi_{\tau}x}_{L^1}.
    \end{align*}
    The assertion is now immediate.
\end{proof}

A subset $S$ of a Banach lattice $X$ is called \emph{almost order-bounded} if for each $\varepsilon>0$ there exists $x_{\varepsilon} \in X_+$ such that
$
    \norm{(\modulus{x}-x_{\varepsilon})_+}<\varepsilon
$
for all $x\in S$. Equivalently, if for each $\varepsilon>0$ there exists $x_{\varepsilon} \in X_+$ such that
\[
    S \subseteq [-x_{\varepsilon},x_{\varepsilon}]+\varepsilon B_X;
\]
where $B_X$ denotes the closed unit ball of $X$. 
Recall that an \emph{AL-space} is a Banach lattice whose norm is additive on the positive cone (cf.\ Table~\ref{table:properties}).

\begin{theorem}
    \label{thm:zero-class-observation-sufficient}
    Let $(T(t))_{t\ge 0}$ be a $C_0$-semigroup on a reflexive Banach space $X$ and let $Y$ be an AL-space.     
    If $C \in \calL(X_1,Y)$ is $L^1$-admissible, then it is also zero-class $L^1$-admissible.
\end{theorem}

\begin{proof}
    Reflexivity of $X$ implies weak compactness of the output operator $\Psi_1$, so $\Psi_1( B_{X})$ is relatively weakly compact. Next, as $Y$ is an AL-space, so is $L^1([0,1], Y)$. 
    Consequently, $\Psi_1( B_{X})$ is almost order-bounded by the Dunford-Pettis theorem \cite[Theorem~2.5.4]{Meyer-Nieberg1991}. 
    Letting $(R_t)_{t\ge 0}$ be as in Lemma~\ref{prop:zero-class-translation}, its
    strong continuity and nilpotency now implies that it converges to $0$ uniformly on $\Psi_1( B_{X})$ as $t\uparrow 1$. Consequently, $C$ is zero-class $L^1$-admissible by Lemma~\ref{prop:zero-class-translation}.
\end{proof}

Note that all of the order assumptions in Theorem~\ref{thm:zero-class-observation-sufficient} are posed on $Y$ whereas $X$ is a general Banach space. We don't know if the Dunford-Pettis theorem \cite[Theorem~2.5.4]{Meyer-Nieberg1991} can be generalised to ordered Banach spaces whose norm are additive on the positive cone. Therefore, the proof above required that $Y$ is a Banach lattice.

\begin{remark}
    In the proof of Theorem~\ref{thm:zero-class-observation-sufficient}, reflexivity was needed solely for $\Psi_1$ to be weakly compact. In other words, for AL-space valued $L^1$-admissible observation operators, weak compactness of the output map implies zero-class $L^1$-admissibility.
\end{remark}

 The following example shows that the reflexivity assumption in Theorem~\ref{thm:zero-class-observation-sufficient} cannot be dropped, in general; see also \cite[Example~26]{JacobSchwenningerZwart2019}.

\begin{example}
    \label{exa:shift-L1-endpoint}
    Consider the left translation semigroup on $X:=L^1[0,1]$ with observation $C:X_1\to \bbC$ given by $f\mapsto f(0)$. The corresponding output operator $\Psi_{\tau}:X_1 \to L^1([0,\tau],Y)$ satisfies $\norm{\Psi_{\tau}}=1$ for  each $\tau>0$. Therefore, $C$ is $L^1$-admissible but the admissibility is not zero-class.
\end{example}

Combining Proposition~\ref{prop:admissibility-observation-additive-norm} with Theorem~\ref{thm:zero-class-observation-sufficient}  yields the following sufficient condition for zero-class $L^1$-admissibility of positive observation operators. Yet another condition for zero-class $L^1$-admissibility of observation operators is given in Lemma~\ref{lem:zero-class-observation-dual}.

\begin{corollary}
    Suppose that Assumption~\ref{ass:general} holds. If $X$ is reflexive, $Y$ is an AL-space and $C \in \calL(X_1,Y)$ is positive, then $C$ is zero-class $L^1$-admissible.
\end{corollary}

\section{Admissibility of control operators}
    \label{sec:admissibility-control}

In this section, we shift our attention to admissibility considerations regarding the system
\begin{equation}
    \label{eq:system-control}
    \left.
        \begin{aligned}
            \dot{x}(t) &= Ax(t)+Bu(t),\quad t\ge 0\\
            x(0)       &= x_0
        \end{aligned}
    \quad\right\}\ \Sigma(A,B),
\end{equation}
where $A$ generates a $C_0$-semigroup $(T(t))_{t\ge 0}$ on a Banach space $X$ and the \emph{control} operator $B\in\calL(U,X_{-1})$ is defined on a Banach space $U$. Most of our results in this section are proved in the following setting:

\begin{assumption}
    \label{ass:generalB}
    Suppose that $X$ is an ordered Banach space with a generating and normal cone and $(T(t))_{t\ge 0}$ is a positive $C_0$-semigroup on $X$. 
    Moreover, let $U$ be a Banach space and $B\in \calL(U,X_{-1})$.
\end{assumption}

In what follows, we let $\mathrm{T}([0,\tau], U)$ denote the space of all $U$-valued step functions on $[0,\tau]$, i.e., a piecewise constant function with finitely many jumps. This is a normed space when equipped with the supremum norm whose closure is called the space of \emph{regulated} functions -- denoted by $\Reg([0,\tau], U)$.
Now, let $\mathrm Z$ be a placeholder for $L^\infty$ or $\mathrm C$ or $\Reg$. Corresponding to the system in~\eqref{eq:system-control}, we define the \emph{input map}
\begin{equation}
    \label{eq:input-operator}
        \Phi_{\tau}:\mathrm{Z}([0,\tau],U)\to       X_{-1},\qquad
                                        u \mapsto   \int_0^{\tau}T_{-1}(\tau-s)Bu(s)\dx s
\end{equation}
for fixed $\tau>0$. 
The operator $B$ is called a \emph{$\mathrm Z$-admissible control operator} if $\Ima \Phi_{\tau} \subseteq X$ for some (equivalently, all) $\tau>0$. Additionally, if  $\limsup_{\tau\downarrow 0} \norm{\Phi_{\tau}}_{\calL(\mathrm{Z}([0,\tau],U), X)}=0$, then we say that $B$ is a \emph{zero-class $\mathrm{Z}$-admissible control operator}. 
The admissibility of control operators was recently generalized to a notion closely related to maximal regularity in \cite[Section~2]{AroraMironchenko2025}

For $\omega \in \bbC$, it is easy to see that $B$ is $\mathrm Z$-admissible for $(T(t))_{t\ge 0}$ if and only if $B$ is $\mathrm Z$-admissible for $(e^{\omega t} T(t))_{t\ge 0}$. For this reason, when proving admissibility of control operators, we often assume (without loss of generality) that $\gbd(A)<0$ thereby having $0\in \resSet(A)$.

 Since $C([0,\tau],U)\subseteq \Reg([0,\tau],U)$, every $\Reg$-admissible control operator is $\mathrm{C}$-admissible. In fact, the converse is also true and seems to be folklore. The proof for the particular case $U=X$ and $B= A_{-1}$ is given in \cite[Proposition~2.2]{JacobSchwenningerWintermayr2022} and the proof for the general case follows mutatis mutandis; see also \cite[Proposition~2.1]{AroraSchwenninger2024}. As we use this observation multiple times in the sequel, we state it explicitly:

\begin{proposition}
    \label{prop:regulated-admissibility-equivalence}
    Let $(T(t))_{t \ge 0}$ be a $C_0$-semigroup on a Banach space $X$, let $U$ be a Banach space, and let $B\in \calL(U, X_{-1})$.
    The operator $B$ is $\mathrm{C}$-admissible if and only if $B$ is $\Reg$-admissible.
\end{proposition}

\subsection{Positivity of $B$ in terms of the boundary control operator}
\label{subsec:positive-B-via-boundary-operator}

As explained in the introduction, the additive input term $Bu$ in~\eqref{eq:system-control} with an operator $B: U \to X_{-1}$ is often used to encode boundary control. 
Since several results in this section use the assumption that $B$ is positive, we will first describe how the positivity of $B$ can be described in terms of the boundary operator. 
The precise setting is as follows.

Let $X$ and $U$ be ordered Banach spaces, let $\dom\mathfrak{A}$ be a vector subspace of $X$ that is a Banach space with respect to a stronger norm, and let $\mathfrak{A}:\dom\mathfrak{A} \to X$ and $\mathfrak{B}:\dom\mathfrak{A}\to U$ be bounded linear operators (but the norm on $\dom \mathfrak{A}$ does not need to be the graph norm of $\mathfrak{A}$ nor does $\mathfrak{A}$ need to be closed as an operator on $X$). 
Assume that $\mathfrak{B}$ is surjective and that the restriction $A:=\mathfrak{A}|_{\ker\mathfrak{B}}$ generates a positive $C_{0}$-semigroup on $X$.
For every $\lambda \in \resSet(A)$, one has the following properties (see e.g.\ \cite[Lemma~1.2]{Greiner1987}): 
the domain $\dom\mathfrak{A}$ is the direct sum of its subspaces $\dom (A) = \ker \mathfrak B$ and $\ker(\lambda - \mathfrak{A})$ and the restriction $\mathfrak{B}|_{\ker(\lambda - \mathfrak{A})}$ is a bijection from $\ker(\lambda - \mathfrak{A})$ to $U$. 
We denote its inverse by $B_\lambda: U \to \ker(\lambda - \mathfrak{A}) \subseteq \dom \mathfrak{A} \subseteq X$ and note that $B_\lambda$ is bounded from $U$ to $\dom \mathfrak{A}$ by the continuous inverse theorem. 
The boundary control problem
\begin{align*}
    \dot{x}(t)={}&\mathfrak{A}x(t),\quad t>0,\quad x(0)=0,\\
    \mathfrak{B}x(t)={}&u(t)
\end{align*}
with control $u(t) \in U$ can, under appropriate assumptions, be reformulated as the system~\eqref{eq:system-control} if one chooses $B := (\lambda-A_{-1})B_\lambda: U \to X_{-1}$, where $A_{-1}$ is the generator of the extrapolated $C_0$-semigroup on $X_{-1}$; see \cite[Proposition~2.8]{Schwenninger2020}. 
One can check that the operator $B$ does not depend on the choice of $\lambda$ (see for instance \cite[Formula~(1.16) in Lemma~1.3]{Greiner1987}). 
We now characterise positivity of $B$ in terms of  $B_\lambda$. 

\begin{proposition}
    \label{prop:positive-B-via-boundary-operator}
    In the setting above, the following assertions are equivalent:
    \begin{enumerate}[\upshape(i)]
        \item  
        The operator $B: U \to X_{-1}$ is positive. 

        \item 
        The operator $B_\lambda: U \to X$ is positive for all sufficiently large $\lambda \in \bbR$.

        \item 
        The operator $B_\lambda: U \to X$ is positive for all $\lambda > \spb(A)$.
    \end{enumerate}
\end{proposition}

\begin{proof}
    ``(i) $\Rightarrow$ (iii)'': 
    Fix $\lambda > \spb(A)$. 
    Then $\Res(\lambda,A_{-1})$ is a positive operator from $X_{-1}$ to $X$ according to Proposition~\ref{prop:order:extrapolation}(c). 
    Hence, $B_\lambda = \Res(\lambda,A_{-1})B$ is positive.
    
    ``(iii) $\Rightarrow$ (ii)'': 
    This implication is obvious. 

    ``(ii) $\Rightarrow$ (i)'': 
    For $u \in U_+$, one has 
    \[
        Bu 
        = 
        \lim_{\lambda \to \infty} \lambda \Res(\lambda,A_{-1}) Bu 
        = 
        \lim_{\lambda \to \infty} \lambda B_\lambda u 
        \in 
        X_{-1,+}
        ,
    \]
    where both limits are norm limits in $X_{-1}$.
\end{proof}

We point out that the equivalence of (i) and (iii) in Proposition~\ref{prop:positive-B-via-boundary-operator} above was also noted in \cite[Lemma~2.1]{Gantouh2022a} and on \cite[Page~16]{Gantouh2024}.
In a similar vein, if one already knows that a boundary control system is admissible, then the positivity of the operators $B_\lambda$ characterises whether positive inputs lead to positive trajectories \cite[Proposition~4.3]{EngelKramar17}. 
The condition that all $B_\lambda$ be positive also occurs in assumption~(ii) of \cite[Theorem~4.3]{Gantouh2022a}.
It is instructive to see what Proposition~\ref{prop:positive-B-via-boundary-operator} says for the Laplace operator.  
The following example demonstrates that the positivity condition on $B_\lambda$ can be interpreted as the maximum principle for the Laplace operator in this case.

\begin{example}
    Let $\Omega \subseteq \bbR^d$ be a bounded smooth domain (say, for simplicity, with $C^\infty$-boundary), set $X := L^p(\Omega)$ for some $p \in (1,\infty)$, and let $\mathfrak{A}$ denote the Laplace operator with domain $W^{2,p}(\Omega)$. 
    Let $U \subseteq L^p(\partial \Omega)$ denote the image of the trace operator on $W^{2,p}(\Omega)$, endowed with the order inherited from $L^p(\partial \Omega)$ and let $\mathfrak{B}: W^{2,p}(\Omega) \to U$ be the trace operator. 

    Then, in the notation of Proposition~\ref{prop:positive-B-via-boundary-operator}, $A$ is the Dirichlet Laplace operator on $L^p(\Omega)$. 
    For $\lambda > \spb(A)$ and $u \in U$, the function $v := B_\lambda u \in W^{2,p}(\Omega)$ solves the boundary value problem
    \begin{align*}
        (\lambda-\Delta) v & = 0, \\ 
        v|_{\partial \Omega} & = u.
    \end{align*}
    So if $p \ge 2$, then it follows from the maximum principle for Sobolev functions \cite[Theorem~8.1]{GilbargTrudinger2001} that the operator $B_\lambda$ is positive. 
    Hence, Proposition~\ref{prop:positive-B-via-boundary-operator} yields that $B$ is positive if $p \ge 2$.
\end{example}

\subsection{Input admissibility for positive semigroups}

If $U$ and $X$ in Assumption~\ref{ass:generalB} are reflexive and if
$B\in \calL(U, X_{-1})$ maps the unit ball of $U$ into an order-bounded subset of $X_{-1}$, then the same assumption holds for the double dual $B''\in \calL(U'', (X'')_{-1})$. Due to the reflexivity of $U$, we obtain that $B'$ is an $L^1$-admissible observation operator by Corollary~\ref{cor:admissibility-observation-order-bounded}. Therefore, reflexivity of $X$ allows us to appeal to the Weiss duality theorem \cite[Theorem~6.9(ii)]{Weiss1989a} to obtain that $B$ is $L^\infty$-admissible. 
Without the reflexivity assumptions on $U$ and $X$, we are able to show that $L^\infty$-admissibility is always zero-class. 

In what follows, we repeatedly use the following property of normal cones. 
If $X$ is an ordered Banach space with a normal cone, then by \cite[Theorem~2.38]{AliprantisTourky2007} there exists $c>0$ such that
\begin{equation}
    \label{eq:normality-equivalent}
    x \in [a,b] \quad \Rightarrow \quad \norm{x} \le c \max\{ \norm a, \norm b\}.
\end{equation}

\begin{theorem}
    \label{thm:zero-class-control-order-bounded}
    Suppose that Assumption~\ref{ass:generalB} holds and that $B\in \calL(U, X_{-1})$ maps the unit ball of $U$ into an order-bounded subset of $X_{-1}$.
    \begin{enumerate}[\upshape (a)]
        \item 
        If $B$ is $L^\infty$-admissible, then it is zero-class $L^\infty$-admissible.
        
        \item 
        If both $X$ and $U$ are reflexive, then $B$ is zero-class $L^\infty$-admissible. 

        \item 
        If $B$ is $\mathrm C$-admissible, then it is zero-class $\mathrm C$-admissible.
    \end{enumerate} 
\end{theorem}

\begin{proof}
    (a) 
    By assumption, there exist $b_1,b_2\in X_{-1}$ such that $B(B_U) \subseteq [b_1,b_2]$.
    Since the corresponding input operator $\Phi_{\tau}:L^\infty([0,\tau], U)\to X_{-1}$ maps into $X$, for each $u\in L^{\infty}([0,\tau],U)$ taking values in $B_U$, we get the inequalities
    \[
        \int_0^{\tau} T_{-1}(\tau-s)b_1\dx s 
        \le 
        \Phi_{\tau}u
        \le 
        \int_0^{\tau} T_{-1}(\tau-s)b_2\dx s
    \]
    in $X$ and hence the norm estimate
    \[ 
        \norm{\Phi_{\tau}u}_X \le c \max_{i=1,2} \norm{\int_0^{\tau} T_{-1}(\tau-s)b_i\dx s}_X,
    \]
    where $c>0$ is such that~\eqref{eq:normality-equivalent} holds.
    Without loss of generality, assume that $\gbd(A)<0$ and that the norm on $X_{-1}$ is given by $\norm{x}_{-1} = \norm{(A_{-1})^{-1}x}_X$ for all $x \in X_{-1}$. It follows that one can estimate $\norm{\Phi_{\tau}u}_X$ from above by
    \[
         \norm{\Phi_{\tau}u}_X \le
                              c \max_{i=1,2}\norm{A_{-1}\int_0^{\tau} T_{-1}(\tau - s)b_i\dx s}_{-1}
                              = c \max_{i=1,2}\norm{T_{-1}(\tau)b_i-b_i}_{-1}
                              \to 0
    \]
    as $\tau\to 0$. In other words, $B$ is zero-class $L^\infty$-admissible. 
    
    (b) 
    This now follows from the discussion at the beginning of the subsection.

    (c)
    The arguments in the proof of~(a) can be repeated mutatis mutandis.
\end{proof}

% \begin{remark}
%     \label{rem:zero-class-order-bounded}
%     With the same arguments as above, it can be shown that Theorem~\ref{thm:zero-class-control-order-bounded}(a) remains true if $L^\infty$ is replaced with $\mathrm{C}$.
% \end{remark}

Dropping the reflexivity assumptions in the second part of Theorem~\ref{thm:zero-class-control-order-bounded} we are able to at least show zero-class $\mathrm C$-admissibility:

\begin{theorem}
    \label{thm:zero-class-C-admissible}
    Suppose that Assumption~\ref{ass:generalB} holds and that $B$ maps the unit ball of $U$ into an order-bounded subset of $X_{-1}$.
    Then $B$ is zero-class $C$-admissible. 
\end{theorem}

\begin{proof}
    Let $\tau > 0$.
    For each step function $u\in \mathrm{T}([0,\tau], U)$, we know that
    \[
        \int_0^{\tau} T_{-1}(\tau-s)Bu(s)\dx s \in X
    \]
    from \cite[Lemma~4.3(ii)]{BatkaiJacobVoigtWintermayr2018}; note that the proof of \cite[Lemma~4.3(ii)]{BatkaiJacobVoigtWintermayr2018} does not use any positivity properties. By assumption, there exist $b_1,b_2\in X_{-1}$ such that $B(B_U) \subseteq [b_1,b_2]$. Choose $c>0$ such that~\eqref{eq:normality-equivalent} holds.
    If $u \in \mathrm{T}([0,\tau], U)$ with $\norm{u}_{\infty}\le 1$, then due to the positivity of the semigroup,
    \[
        \norm{\int_0^{\tau} T_{-1}(\tau-s)Bu(s)\dx s} \le c \max_{i=1,2} \norm{\int_0^{\tau} T_{-1}(\tau-s)b_i\dx s} =: K_{\tau}
    \]
    As a result, 
    \[
        \norm{\int_0^{\tau} T_{-1}(\tau-s)Bu(s)\dx s} \le K_{\tau} \norm{u}_{\infty}
    \]
    for all step functions $u \in \mathrm{T}([0,\tau], U)$. Hence, $B$ is $\mathrm{C}$-admissible (see, \cite[Remark~4.1.7]{Wintermayr2019} and Proposition~\ref{prop:regulated-admissibility-equivalence}). 
    The zero-class $\mathrm C$-admissibility now follows from  Theorem~\ref{thm:zero-class-control-order-bounded}(c).
\end{proof}

In particular, the following generalisation of \cite[Corollary~4.2.10]{Wintermayr2019} follows by combining Theorem~\ref{thm:zero-class-C-admissible} and Proposition~\ref{prop:unit-ball-into-order-bounded-automatic}(a).

\begin{corollary}
    \label{cor:zero-class-C-admissible}
    Suppose that Assumption~\ref{ass:generalB} holds and $U$ is an ordered Banach space with a unit.
    If $B\in \calL(U, X_{-1})$ is positive, then it is zero-class $\mathrm C$-admissible. 
\end{corollary}

Proposition~\ref{prop:unit-ball-into-order-bounded-automatic}(b) gives another situation where $B$ satisfies the assumptions of Theorem~\ref{thm:zero-class-C-admissible}. However, such a condition even yields $L^\infty$-admissibility (Corollary~\ref{cor:zero-class-order-bounded-automatic}).

\begin{remark}
    \label{rem:kato-example}
    The assumption that $U$ has a unit cannot be dropped in Corollary~\ref{cor:zero-class-C-admissible}. Indeed, on the Banach lattice $c_0$, there is an example \cite[Example~2.3]{JacobSchwenningerWintermayr2022}, going back to Kato, of a positive semigroup whose generator is positive, $\mathrm C$-admissible, yet it is not zero-class $\mathrm C$-admissible. 
\end{remark}

Recall from Table~\ref{table:properties} that an AM-space is a Banach lattice whose open unit ball is upwards directed. In particular, $c_0$ is an AM-space. It was shown in \cite[Theorem~4.1.19]{Wintermayr2019} that if $U$ is an AM-space, then every positive $B\in\calL(U, X_{-1})$ is at least $\mathrm C$-admissible. We generalise this in the following theorem.
    
\begin{theorem}
    \label{thm:admissibility-control-upwards-directed}
    Suppose that Assumption~\ref{ass:generalB} holds and $U$ is an ordered Banach space whose open unit ball is upwards directed.
    If $B\in \calL(U, X_{-1})$ is positive, then it is $\mathrm C$-admissible. 
\end{theorem}

\begin{proof}
    Without loss of generality, we assume that $\gbd(A)<0$.
    Let $\Phi_{\tau}$ be the input operator defined in~\eqref{eq:input-operator} with $Z=\Reg$. As in the proof of Theorem~\ref{thm:zero-class-C-admissible}, $\Phi_{\tau}$ maps the space of step functions $\mathrm{T}([0,\tau], U)$  into $X$.

    Let $u\in\mathrm{T}([0,\tau], U)$. Then we can write $u= \sum_{k=1}^n u_k \chi_{I_k}$ for $u_k\in U$ and disjoint intervals $I_k \subseteq [0,\tau]$. Since the inerior of $B_U$ is upwards directed, the norm on $U'$ is additive on the positive cone $(U')_+$ by \cite[Lemma~3]{Ng1969}. Whence, according to \cite[Theorem~1.3(2)]{Messerschmidt2015}, there exist $\alpha > 1$ -- independent of $u$ -- and $ w_1,w_2\in U$ such that $-u_k\le w_1$ and $u_k\le w_2$ for each $k$ and $\norm{w_i}\le \alpha \norm{u}_\infty$. In particular, $-w_1\le u(s)\le w_2$ for all $s \in [0,\tau]$.
    
    Due to normality of the cone, we can choose $c>0$ such that~\eqref{eq:normality-equivalent} holds. 
    Positivity of the semigroup, the operator $B$, and the elements $w_1,w_2\in U_+$ now allow us to estimate $\norm{\int_0^{\tau} T_{-1}(\tau-s) Bu(s)\dx s}$ from above by
    \[
        c\max_{i=1,2}\norm{\int_0^{\infty} T_{-1}(\tau-s)Bw_i\dx s} \le c\norm{A_{-1}^{-1}B}\max_{i=1,2} \norm{w_i} \le c\ \alpha \norm{A_{-1}^{-1}B} \norm{u}_\infty.
    \]
    It follows that $B$ is $\Reg$-admissible by \cite[Remark~4.1.7]{Wintermayr2019}. The $\mathrm C$-admissibility is therefore true due to Proposition~\ref{prop:regulated-admissibility-equivalence}.
\end{proof}

Since the open unit ball is upwards directed if and only if the dual norm is additive on the positive cone \cite[Proposition~7 and Theorem~8]{TzschichholtzWeber2005}, Proposition~\ref{prop:admissibility-observation-additive-norm} and Theorem~\ref{thm:admissibility-control-upwards-directed} are dual to each other.

\begin{theorem}
   \label{thm:zero-class-control-order-bounded-automatic} 
    Suppose that Assumption~\ref{ass:generalB} holds and assume that $X_+$ is a face of $X''_+$. 
    If $B$ maps the unit ball of $U$ into an order-bounded subset of $X_{-1}$, then 
    $B$ is zero-class $L^{\infty}$-admissible. 
\end{theorem}

\begin{proof}
    By Theorem~\ref{thm:zero-class-control-order-bounded}, it suffices to show that $B$ is $L^{\infty}$-admissible.
    By assumption, there exists $a,b\in X_{-1}$ such that $B(B_U) \subseteq [a,b]$.
    First of all, note that
    $\int_{0}^{\tau} T_{-1}(\tau-s)a\dx s,\int_{0}^{\tau} T_{-1}(\tau-s)b\dx s\in \dom A_{-1}=X$.
    Now for each $u\in L^{\infty}([0,\tau],U)$ with $\norm{u}_{\infty}\le 1$, we have
    \[
       \int_{0}^{\tau} T_{-1}(\tau-s)a\dx s\le \Phi_{\tau} u \le \int_{0}^{\tau} T_{-1}(\tau-s)b\dx s.
    \]
    Because $X_+$ is a face of $X''_+$, also $X_+$ is a face of $X_{-1,+}$ by Corollary~\ref{cor:face-in-extrapolation-space}. It follows that $\Phi_{\tau} u\in X$. A rescaling argument yields that $\Phi_{\tau} u\in X$ for all $u\in L^{\infty}([0,\tau],U)$.
\end{proof}

Assumptions analogous to Theorem~\ref{thm:zero-class-control-order-bounded-automatic} were recently used to provide a duality result in \cite[Theorem~4.4]{AroraSchwenninger2024}.
From Theorem~\ref{thm:zero-class-control-order-bounded-automatic} and Proposition~\ref{prop:unit-ball-into-order-bounded-automatic} we get that zero-class $L^\infty$-admissibility is automatic in the following situations.

 \begin{corollary} 
    \label{cor:zero-class-order-bounded-automatic}
    Let Assumption~\ref{ass:generalB} hold, assume that $U$ is an ordered Banach space, and let $B \in \calL(U, X_{-1})$ be positive.
    Each of the following assumptions is sufficient for $B$ to be zero-class $L^{\infty}$-admissible.
    \begin{enumerate}[\upshape (a)]
        \item The space $U$ has a unit and $X_+$ is a face of $X''_+$.
            
        \item The open unit ball of $U$ is upwards directed and $X$ is a KB-space.
    \end{enumerate}
\end{corollary}

\begin{proof}
    In each case, $X_+$ is a face of $X''_+$ \cite[Theorem~7.1]{Wnuk1999}, so by Theorem~\ref{thm:zero-class-control-order-bounded-automatic}, we only need to show that $B$ maps  the unit ball of $U$ into an order-bounded subset of $X_{-1}$. 
    For (a), this is shown in Proposition~\ref{prop:unit-ball-into-order-bounded-automatic}(a).
    
    (b) Due to Corollary~\ref{cor:inherited-properties-of-extrapolation-space}(a),  every increasing norm-bounded net in $X_{-1, +}$ is norm-convergent. The claim now follows by Proposition~\ref{prop:unit-ball-into-order-bounded-automatic}(b).
\end{proof}

Here again, the example \cite[Example~2.3]{JacobSchwenningerWintermayr2022} by Kato mentioned in Remark~\ref{rem:kato-example} shows that the assumption that $U$ has a unit cannot be dropped in Corollary~\ref{cor:zero-class-order-bounded-automatic}(a).

\subsection{Input admissibility for general semigroups}

In this section, we leave the setting of Assumption~\ref{ass:generalB} and instead impose order assumptions on the input space. This allows us to give conditions under which admissibility of the control operator automatically gives zero-class $L^\infty$-admissibility. 
First, we obtain the following analogue to Theorem~\ref{thm:zero-class-observation-sufficient}:

\begin{theorem}
    Let $(T(t))_{t\ge 0}$ be a $C_0$-semigroup on a reflexive Banach space $X$ and let $U$ be an AM-space. 
    If $B \in \calL(U, X_{-1})$ is $L^\infty$-admissible, then it is also zero-class $L^\infty$-admissible.
\end{theorem}

\begin{proof}
    Since $B$ be $L^\infty$-admissible, $B'$ is $L^1$-admissible \cite[Theorem~6.9(iii)]{Weiss1989a}. Since $U$ is an AM-space, we have that $U'$ is an AL-space. Moreover, as $X$ is reflexive, so is $X'$. Thus, Theorem~\ref{thm:zero-class-observation-sufficient} implies that $B'$ is zero-class $L^1$-admissible.  Keeping in mind that $X$ is reflexive, the result follows by \cite[Theorem~6.9(ii)]{Weiss1989a}.
\end{proof}

In fact, if $U$ is finite-dimensional, then even $\mathrm C$-admissibility of the control operator implies that it is zero-class $L^{\infty}$-admissibility:

\begin{theorem}
    \label{thm:zero-class-control-finite-dimensions}
    Let $(T(t))_{t\ge 0}$ be a $C_0$-semigroup on a reflexive Banach space $X$ and let $U$ be finite-dimensional. 
    If $B \in \calL(U, X_{-1})$ is $\mathrm C$-admissible, then it is also zero-class $L^\infty$-admissible.
\end{theorem}

The proof uses the following duality lemma in the spirit of Weiss \cite[Theorem~6.9]{Weiss1989a}:

\begin{lemma}
    \label{lem:dual-continuous}
    Let $(T(t))_{t\ge 0}$ be a $C_0$-semigroup on a reflexive Banach space $X$ and let $U$ be finite-dimensional.
    If $B\in \calL(U, X_{-1})$ is a $\mathrm C$-admissible control operator,  then $B' \in \calL((X')_1, U)$ is an $L^1$-admissible observation operator.
\end{lemma}

\begin{proof}
    Let $\tau>0$. Since $B$ is $\mathrm{C}$-admissible, there exists $K_{B,\tau}>0$ such that 
    $
        \norm{\int_{0}^{\tau} T_{-1}(t)Bu(t)\dx t  }\le K_{B,\tau} \norm{u}_{\infty}
    $
    for all $u\in \mathrm{C}([0,\tau],U)$.
    Finite-dimensionality of the input space allows us to compute
    \begin{align*}
        \norm{B'T'(\argument)x'}_1 & = \sup_{u \in \mathrm{C}([0,\tau],U), \norm{u}_{\infty}\le 1} \modulus{\int_0^{\tau} \duality{B'T'(t)x'}{u(t)} \dx t}\\
                                            & = \sup_{u \in \mathrm{C}([0,\tau],U), \norm{u}_{\infty}\le 1} \modulus{\int_0^{\tau} \duality{B'(T_{-1})'(t)x'}{u(t)} \dx t}\\
                                            & = \sup_{u \in \mathrm{C}([0,\tau],U), \norm{u}_{\infty}\le 1} \modulus{ \duality{x'}{\int_0^{\tau}T_{-1}(t)Bu(t)\dx t} }
                                             \le K_{B,\tau} \norm{x'}
    \end{align*}
    for all $x'\in \dom A'$. This proves that $B'$ is $L^1$-admissible.
\end{proof}

\begin{proof}[Proof of Theorem~\ref{thm:zero-class-control-finite-dimensions}]
    Since $B$ be $\mathrm C$-admissible, $B'$ is zero-class $L^1$-admissible by Lemma~\ref{lem:dual-continuous} and Theorem~\ref{thm:zero-class-observation-sufficient}. Once again, using the reflexivity of $X$, it follows that  $B$ is zero-class $L^\infty$-admissible by the Weiss duality result \cite[Theorem~6.9(ii)]{Weiss1989a}.
\end{proof}

\subsection{$L^r$-input admissibility of positive semigroups}
    \label{sec:L^r-admissibility}

We close this section by moving away from the limit-cases of $L^\infty$- and $\mathrm C$-admissibility and instead considering $L^r$-admissibility for $r < \infty$. 
If one has a sufficiently strong estimate on the resolvent one can even get $L^1$-admissibility of positive input operators \cite[Theorems~2.1]{Gantouh2022a}; compare however Appendix~\ref{appendix:al-characterisation}.
In the following corollary, we instead work with an intrinsic ultracontractivity type assumption on the semigroup. In Theorem~\ref{thm:ultracontractive}(b), this same condition was shown to imply that $\linSpan X_{-1,+}$ is contained in a Favard space. 

\begin{corollary}
    \label{cor:Lr-admissibility}
    Let $(T(t))_{t\ge 0}$ be a positive and analytic $C_0$-semigroup on a reflexive ordered Banach space $X$ whose cone $X_+$ is generating.
    Let $\varphi \in X'$ be strictly positive, assume that $\dom A' \cap (X')_\varphi$ is majorizing in $(X')_\varphi$, that $T(t)' X' \subseteq (X')_\varphi$ for each $t > 0$ and that the ultracontractivity type estimate $\norm{T(t)'}_{X' \to (X')_{\varphi}} \le \widetilde c t^{-\alpha}$ holds for some $\widetilde c \ge 0$ and $\alpha \in [0,1)$ and for all $t \in (0,1]$.

    If $U$ is an ordered Banach space and $B \in \calL(U, X_{-1})$ is positive (more generally, the difference of two positive operators), then $B$ is $L^r$-admissible for every $r > \frac{1}{1-\alpha}$.
\end{corollary}

\begin{proof}
    By Theorem~\ref{thm:ultracontractive}(b), $B$ maps into the Favard space $F_{-\alpha}$, whence, by the closed graph theorem, $B\in\calL(U,F_{-\alpha})$. This readily implies that $B$ is $L^r$-admissible for $r > \frac{1}{1-\alpha}$, see e.g.\ \cite[Lemma 2.1 and Remark 2.3]{PreusslerSchwenninger2024} or \cite[Proposition 19]{Maragh2014}.
\end{proof}

\begin{example}
    Let $\Omega \subseteq \bbR^d$ be a bounded domain with Lipschitz boundary, let $p,q \in (1,\infty)$ be Hölder conjugates, and let $A$ be the elliptic operator with Neumann boundary conditions on $L^p(\Omega)$ from Example~\ref{exa:X-1-for-neumann}. 
    If $\alpha := \frac{d}{2q} < 1$, then it follows from~(a) in Example~\ref{exa:X-1-for-neumann} that $\linSpan L^p(\Omega)_{-1,+}$ is contained in the Favard space $F_{-\alpha}$. 
    
    Now let $U$ be an ordered Banach space and let $B \in \calL(U, X_{-1})$ be positive or, more generally, the difference of two positive operators. 
    Then $B$ is $L^r$-admissible for every $r \in (\frac{1}{1-\alpha},\infty]$ by Corollary~\ref{cor:Lr-admissibility}. 
    If $s \in [1,\infty)$ denotes the Hölder conjugate of $r$, this can be reformulated by saying that $B$ is $L^r$-admissible if $s < \frac{2q}{d} = \frac{1}{\alpha}$.

    We note that a fixed positive operator $B$ can be $L^r$-admissible for a larger range of $r$. 
    For instance, $B$ obtained from the Neumann boundary control is $L^r$-admissible for $r>4/3$, if the boundary of $\Omega$ is $C^2$, see e.g.\ \cite{Lasiecka} or \cite[Example~2.14]{Schwenninger2020}, and even for $r\ge4/3$ when the boundary is $C^{\infty}$ \cite[Proposition~2.4]{HaakKunstmann2008} and \cite[Remark 2.8]{PreusslerSchwenninger2024}.
\end{example}

\section{Applications to perturbation results}
    \label{sec:perturbations}

The consequences of admissibility for perturbation results are well-known. Indeed, if $(T(t))_{t \ge 0}$ is a $C_0$-semigroup on a Banach space $X$ and $B\in \calL(X, X_{-1})$ is zero-class $\mathrm C$-admissible, then the restriction of $A_{-1}+B$ to $X$ generates a $C_0$-semigroup on $X$ \cite[Corollary~III.3.3]{EngelNagel2000}. As a consequence, perturbation results for positive $C_0$-semigroups on AM-spaces were proved in \cite{BatkaiJacobVoigtWintermayr2018}. Likewise, zero-class $L^1$-admissibility of the observation operator also yields a perturbation result \cite[Corollary~III.3.16]{EngelNagel2000}. Using results obtained in the prequel, we are thus able to obtain perturbation results for positive $C_0$-semigroups.

\begin{corollary}
    \label{cor:pert-order-bdd}
    Let $X$ be an ordered Banach space with a generating and normal cone and let $A$ be the generator of a positive $C_0$-semigroup on $X$.
    
    If $B \in \calL(X,X_{-1})$ maps the unit ball of $X$ to an order-bounded set in $X_{-1}$,
    then the part of $A_{-1}+B$ in $X$ generates a $C_0$-semigroup on $X$. 
    If $B$ is positive, so is the perturbed semigroup.
\end{corollary}

\begin{proof}
    The assumptions imply that $B$ is zero-class $\mathrm C$-admissible (Theorem~\ref{thm:zero-class-C-admissible}). 
    Hence, the claim follows from the perturbation result in \cite[Corollary~III.3.3]{EngelNagel2000}. 
\end{proof}

As a consequence of Corollary~\ref{cor:pert-order-bdd}, one can see at once that the operator $A_{-1}+B$ in \cite[Example~5.1]{BatkaiJacobVoigtWintermayr2018} generates a $C_0$-semigroup without verifying any spectral conditions and one can also allow multiplication operator with function in $L^1$ as perturbations:

\begin{example}
    On the space $X=\{ f\in C[0,1]:f(1)=0\}$ the operator $A$ given by
    \[
        \dom A :=\{ f \in C^1[0,1]:f(1)=f'(1)=0\}, \qquad
              Af := f' 
    \]
    generates a positive $C_0$-semigroup \cite[Example~II.3.19(i)]{EngelNagel2000}  that corresponds to
    the partial differential equation
    \begin{align*}
        & &  u_t(t,x) & = u_x(t,x)   & \text{for } & x\in [0,1], \; t \ge 0,  & & \\
        & &  u(0,x)   & = u_0(x)     & \text{for } & x \in [0,1],          & & \\
        & &  u(t,1)   & = 0          & \text{for } & t \ge 0.              & &
    \end{align*}
    Let $\mu$ be a finite continuous positive Borel measure on $(0,1)$ and let $m\in L^1(0,1)$. It can be shown using standard arguments that $\mu,m \in X_{-1}$ (see, for instance, \cite[Page~344]{BatkaiJacobVoigtWintermayr2018}). In particular, the rank-one operator  $f\mapsto B_1f:= \int_0^1 f(x)\dx x\cdot \mu$ and the multiplication operator $f\mapsto B_2f:= mf$ both lie in $\calL(X,X_{-1})$.
    In fact, $B_1$ and $B_2$ map the unit ball into the order-bounded subsets
    $[-\mu,\mu]$ and $\big[-\modulus{m},\modulus{m}\big]$ of $X_{-1}$ respectively.
    It follows by Corollary~\ref{cor:pert-order-bdd} that the part of $A_{-1}+B_1$ and $A_{-1}+B_2$ in $X$ generate $C_0$-semigroups on $X$.
\end{example}

Either by directly dualizing Corollary~\ref{cor:pert-order-bdd} or by using the results from Section~\ref{sec:admissibility-control} and then dualizing one also gets the following perturbation result:

\begin{corollary}
    \label{cor:pert-dual}
    Let $X$ be a reflexive ordered Banach space with a generating and normal cone and let $A$ be the generator of a positive $C_0$-semigroup on $X$.
    
    If $C\in \calL(X_1, X)$ is such that $C'$ maps the unit ball of $X'$ to an order-bounded set in $(X')_{-1}$, then  $A+C$ generates a $C_0$-semigroup on $X$. 
    If $C$ is positive, so is the perturbed semigroup.
\end{corollary}

The perturbation assertions of Corollary~\ref{cor:pert-dual} hold once we establish that $C$ is zero-class $L^1$-admissible \cite[Corollary~III.3.16]{EngelNagel2000}. Let us prove this separately in the spirit of Section~\ref{sec:zero-class-output}:

\begin{lemma}
    \label{lem:zero-class-observation-dual}
    Let $X$ be a reflexive ordered Banach space with a generating and normal cone, let $(T(t))_{t\ge 0}$ be a positive $C_0$-semigroup on $X$, and let $Y$ be a Banach space. If $C\in \calL(X_1,Y)$ is such that $C'$ maps the unit ball of $Y'$ to an order-bounded set in $(X')_{-1}$, then $C$ is zero-class $L^1$-admissible.
\end{lemma}

\begin{proof}
    Since $X$ is reflexive, $X_+$ coincides with $X''_+$. Therefore, $C'\in \calL(Y', X'_{-1})$ is zero-class $L^\infty$-admissible by Theorem~\ref{thm:zero-class-control-order-bounded-automatic}. Zero-class $L^1$-admissibility of $C$ is now a consequence of Weiss duality result \cite[Theorem~6.9(a)]{Weiss1989a}.
\end{proof}

Note that Corollary~\ref{cor:pert-dual} implies in particular the following result: every positive finite rank perturbation $C: X_1 \to X$ (and more generally, every finite rank perturbation that can be written as the difference of two positive operators $X_1 \to X$) of a positive $C_0$-semigroup on a reflexive Banach lattice again generates a $C_0$-semigroup. 
In fact, this is known to be true even if $X$ is not reflexive and goes back to Arendt and Rhandi \cite[Corollary~2.4]{ArendtRhandi1991}.
Note that this is not true, in general, for finite rank perturbations that do not satisfy any positivity assumption: 
Desch and Schappacher proved in \cite[Theorem~2]{DeschSchappacher1988} that a $C_0$-semigroup on a Banach space $X$ is automatically analytic if every rank-$1$ perturbation $C: X_1 \to X$ still generates a $C_0$-semigroup.

\subsection*{Acknowledgements} 
The first author was funded by the Deutsche Forschungsgemeinschaft (DFG, German Research Foundation) -- 523942381.
The article was initiated during a pleasant visit of the first author to the third author at Tampere University, Finland in autumn 2023. 
The first author is indebted to COST Action 18232 for financial support for this visit. 
The research of the third author was supported by the Research Council of Finland grant number 349002.

\appendix

\section{Order-bounded images of the unit ball}
    \label{sec:factorisation}
In Section~\ref{sec:admissibility-control}, the property that an operator $T\in \calL(U,Z)$ maps the unit ball of a Banach space $U$ into an order-bounded subset of an ordered Banach space $Z$ played an important role in obtaining sufficient conditions for admissibility. 
In this appendix, we give a characterisation and, for positive $T$, a number of sufficient conditions for this property.
Note that if $Z$ is even a Banach lattice, the weaker property that $T$ maps the unit ball of $U$ into an order-bounded subset of $Z''$ is called \emph{majorizing} \cite[Proposition~IV.3.4]{Schaefer1974}. 

Before stating our first result, we clarify that by a \emph{null sequence} in a Banach space, we mean a sequence converging to $0$. Additionally, we make the following simple observation:~let $S$ be an order-bounded subset of an ordered Banach space $X$ such that $0\in S$. Then there exists $x,z\in X_+$ such that $S\subseteq [-x,z]$. Taking $e=x+z\in X_+$, we get that $S$ is contained in $[-e,e]$.

\begin{proposition}
    \label{prop:unit-ball-into-order-bounded}
    Let $U$ be a Banach space, let $Z$ be an ordered Banach space with a normal cone, and let $T\in \calL(U,Z)$.
    The following are equivalent.
    \begin{enumerate}[\upshape (i)]
        \item The operator $T$ maps the unit ball of $U$ into an order-bounded subset of $Z$.
        \item There exists an ordered Banach space $\widetilde X$ with a unit such that $T$ factorises as
        \[
            T: U \overset{T_1}{\longrightarrow} \widetilde X \overset{T_2}{\longrightarrow} Z
        \]
        for an operator $T_1 \in \calL(U,\widetilde X)$ and a positive operator $T_2 \in \calL(\widetilde X, Z)$. 
    \end{enumerate}
    If $Z$ is even a KB-space, then {\upshape (i)} and {\upshape (ii)} are also equivalent to each of the following:
    \begin{enumerate}[resume]
        \item There exists $c>0$ such that for every null sequence $(u_n)$ in the open unit ball of $U$, there exists $z \in Z_+$ such that $\norm{z}\le c$ and $\modulus{Tu_n}\le z$ for all $n\in\bbN$.
        \item There exists $c>0$ such that $\norm{\sup_{u \in F} \modulus{Tu}}\le c \sup_{u \in F}\norm{u}$ for every finite subset $F\subseteq U$.
    \end{enumerate}
\end{proposition}

\begin{proof}
    ``(ii) $\Rightarrow$ (i)'':
    Assume that~(ii) holds. Since $\widetilde X$ has a unit, there exists -- according to \cite[Proposition~2.11]{GlueckWeber2020} -- an element $e\in \widetilde X$ and $\varepsilon>0$ such that the implication
    \[
        \norm{x}\le \varepsilon \quad \Rightarrow \quad x\le e
    \]
    holds for all $x \in \widetilde X$. 
    In particular, taking $\lambda = \varepsilon^{-1}\norm{T_1}$, it follows that $T_1u\in [-\lambda e, \lambda e]$ for every $u\in B_U$. 
    Employing the positivity of $T_2$, it follows that $T=T_2T_1$ maps the unit ball of $U$ into an order-bounded subset of $Z$.

    ``(i) $\Rightarrow$ (ii)'':
    The observation about order bounded subsets made prior to the proposition guarantees the existence of $e\in Z_+$ such that $T(B_U)\subseteq [-e,e]$.
    On the other hand, because the cone of $Z$ is normal,  $\widetilde X:= Z_e$ is an ordered Banach space with a unit.
    The above argument allows us to infer that  $\Ima T\subseteq \widetilde X$.
    Whence, $T$ factors through $\widetilde X$ via the canonical embedding $\widetilde X \hookrightarrow Z$ (which is positive). 

    Finally, suppose that $Z$ is a KB-space. 
    Then $Z_+$ is a face of $(Z'')_+$ according to \cite[Theorem~7.1]{Wnuk1999}. 
    The equivalence of (i)--(iv) now follows from \cite[Proposition~IV.3.4]{Schaefer1974}.
\end{proof}

\begin{proposition}
    \label{prop:unit-ball-into-order-bounded-automatic}
    Let $U$ and $Z$ be ordered Banach spaces. Then $T\in \calL(U,Z)_+$
     maps the unit ball of $U$ into an order-bounded subset of $Z$ in the following cases:
    \begin{enumerate}[\upshape (a)]
        \item The space $U$ has a unit.
            
        \item The open unit ball of $U$ is upwards directed and every increasing norm-bounded net in $Z_+$ is norm-convergent.
    \end{enumerate}
\end{proposition}

\begin{proof}
    If $U$ has a unit, then due to the positivity of $T$, the assumptions in condition~(ii) of Proposition~\ref{prop:unit-ball-into-order-bounded} are fulfilled. Hence, $T$ maps the unit ball of $U$ into an order-bounded subset of $Z$; observe that the proof of (ii) $\Rightarrow$ (i) in Proposition~\ref{prop:unit-ball-into-order-bounded} did not need the cone of $Z$ to be normal.
    
    (b) Let $C$ denote the open unit ball of $U$. 
    Since $C$ is directed, $(Tu)_{u \in C}$ is an increasing and norm bounded net in $Z$ and thus, by assumption, norm converges to an element $z \in Z$. 
    Hence, $TC \subseteq [-z,z]$ and since order intervals are closed, we conclude that $T$ also maps the closed unit ball $B_U = \overline{C}$ into $[-z,z]$.
\end{proof}

\section{Lower resolvent bounds and a characterisation of AL-spaces}
    \label{appendix:al-characterisation}
Let $X$ be a Banach lattice and let $A: \dom A \subseteq X \to X$ be a \emph{resolvent positive} operator, i.e., all sufficiently large real numbers $\lambda$ are in the resolvent set of $A$ and satisfy $\Res(\lambda,A) \ge 0$. 
In particular, this is satisfied if $A$ generates a positive $C_0$-semigroup on $X$.
Consider the following condition: 
\begin{align}
    \label{cond:lower-bound}
    \begin{split}
        \text{There exist numbers } \lambda > & \spb(A) \text{ and } c > 0 \text{ such that } A \\
        \text{ satisfies }  \norm{\Res(\lambda,A)f} & \ge c \norm{f} \text{ for every } f \in X_+.
    \end{split}
\end{align} 
This condition occurs in recent papers on infinite-dimensional positive systems, see \cite[Theorems~2.1]{Gantouh2022a} and \cite[Theorem~1]{Gantouh2024}) and was shown to have rather strong consequences for admissibility of control operators. 
Earlier~\eqref{cond:lower-bound} was used to show generation results \cite[Theorem~2.5]{Arendt1987}. 

The condition~\eqref{cond:lower-bound} is, of course, satisfied in the trivial case where $A$ is bounded. 
More interestingly, if $X$ is an $L^1$-space and the semigroup generated by $A$ is \emph{stochastic} -- i.e., positive and norm preserving on $X_+$ -- then it is also easy to see that~\eqref{cond:lower-bound} is satisfied. 
However, it is not clear at all how to find an unbounded operator on an $L^p$-space for $p > 1$ that satisfies~\eqref{cond:lower-bound}. 
In this appendix we justify this theoretically:~in the important special case where $A$ has compact resolvent, the condition~\eqref{cond:lower-bound} implies that $X$ is -- up to an equivalent renorming -- an $L^1$-space (Corollary~\ref{cor:lower-estimate-compact}).

Let us recall the following terminology again: 
An \emph{AL-space} is a Banach lattice whose norm is additive on the positive cone. 
Every $L^1$-space is an AL-space and conversely, every AL-space can be shown to be isomorphic (as a Banach lattice) to an $L^1$-space (over a possible non-$\sigma$-finite measure space), cf.\ Table~\ref{table:properties}.
We first show a characterisation of AL-spaces, up to equivalent norms. For related geometric results, see \cite[Proposition~3.11]{MarinacciMontrucchio2012} and \cite[Section~2]{AdlyErnstThera2004}.

\begin{theorem}
    \label{thm:char-al}
    For a Banach lattice $X$, the following are equivalent:
    \begin{enumerate}[\upshape (i)]
	\item
	There exists an equivalent norm on $X$ which turns $X$ into an AL-space.
		
	\item
	No net in the positive unit sphere of $X$ converges weakly to $0$.
    \end{enumerate}
\end{theorem}

\begin{proof}
    ``(i) $\Rightarrow$ (ii)'': 
    This can easily be seen by testing against the norm functional on the AL-space.
	
    ``(ii) $\Rightarrow$ (i)'':
    Let $\calF'$ denote the set of all non-empty finite subsets of the positive unit sphere in $X'$. 
    Assume for a moment that the following property~$(*)$ holds: 
    For each number $\varepsilon > 0$ and each $F' \in \calF'$,
    there exists a vector $x_{\varepsilon, F'}$ in the positive unit sphere of $X$ such that 
    $\langle x', x_{\varepsilon, F'} \rangle < \varepsilon$ for all $x' \in F'$. 
    The relation $\preceq$ on $(0,\infty) \times \calF'$  given by 
    $(\varepsilon_1, F_1') \preceq (\varepsilon_2, F_2')$ if and only if $\varepsilon_1 \ge \varepsilon_2$ and $F_1' \subseteq F_2'$, 
    turns $(0,\infty) \times \calF'$ into a directed set. 
    The net $(x_{\varepsilon, F'})_{(\varepsilon, F') \in (0,\infty) \times \calF'}$ 
     in the positive unit sphere of $X$ converges weakly to $0$,
    contradicting~(ii), so the property~$(*)$ cannot  hold.
	
    Hence, there exists a number $\varepsilon > 0$ and a set $F' \in \calF'$ such that 
    for all $x$ in the positive unit sphere of $E$ one has $\langle x', x \rangle \ge \varepsilon$ 
    for some $x' \in F'$.
    Setting $z' := \sum_{x' \in F'} x'$, we obtain $\langle z', x \rangle \ge \varepsilon \norm{x}$ 
    for each $x \in X_+$.
    Hence, $x \mapsto \langle z', \modulus{x} \rangle$ is an equivalent lattice norm on $X$,
    which is clearly additive on the positive cone.
\end{proof}

\begin{corollary}
    \label{cor:lower-estimate-compact}
    Let $X$ be a Banach lattice and let $T: X \to X$ be a compact linear operator. 
    If there exists a number $c > 0$ such that $\norm{Tx} \ge c \norm{x}$ for all $x \in X_+$, 
    then there exists an equivalent norm on $X$ which turns $X$ into an AL-space.
\end{corollary}

\begin{proof}
    If the conclusion is false, then
    by Theorem~\ref{thm:char-al}, there is a net $(x_j)$ in the positive unit sphere of $X$ 
    converging weakly to $0$. 
    From compactness of $T$, we deduce that
    $(Tx_j)$ converges to $0$ in norm, contradicting that
    $\norm{Tx_j} \ge c$ for each index $j$.
\end{proof}

We find it instructive to state the following special case of Corollary~\ref{cor:lower-estimate-compact}.

\begin{corollary}
    Let $X$ be a Banach lattice and let $A: \dom A \subseteq X \to X$ be a resolvent positive linear operator. 
    If $A$ has compact resolvent and satisfies~\eqref{cond:lower-bound}, then there exists an equivalent norm on $X$ which turns $X$ into an AL-space.
\end{corollary}

\bibliographystyle{plainurl} 
\bibliography{literature}

\end{document}